\documentclass[8pt, a4paper]{article}

\usepackage{amsmath,amsfonts,amssymb}
\usepackage{bbm}

\usepackage{lipsum}

\usepackage{xcolor}
\usepackage{ulem}

\newtheorem{definition}{Definition}[section]
\newtheorem{theorem}[definition]{Theorem}
\newtheorem{lemma}[definition]{Lemma}
\newtheorem{proposition}[definition]{Proposition}
\newtheorem{corollary}[definition]{Corollary}
\newtheorem{remark}[definition]{Remark}

\newenvironment{proof}{{\bf Proof }}{{\vskip 0.1cm \hfill$\Box$}}

\begin{document}


\noindent
{\Large \bf Uniform approximation by harmonic polynomials for solving the Dirichlet problem of Laplace’s equation on a disk}
\\ \\
\bigskip
\noindent
{\bf Haesung Lee}  \\
\noindent
{\bf Abstract.} 
In this paper, we study the Dirichlet problem for Laplace’s equation in an open disk. The uniqueness of solutions is ensured by the well-known weak maximum principle. We introduce a novel approach to demonstrate the existence of a solution using harmonic polynomials that converge uniformly to a solution. Specifically, we rigorously derive the convergence rate of the harmonic polynomials and show that smoother boundary data and proximity of the target point to the disk's origin accelerate the convergence. Additionally, we obtain uniform estimates for the derivatives of solutions of arbitrary orders, controlled by $L^1$-boundary data. Notably, the constants in our estimates are significantly improved compared to existing results. Furthermore, we provide a refined convergence region for Taylor's series of the solution, along with error estimates within this region, at each point in the open disk.


\section{Introduction}
Harmonic functions play an important role in mathematics, physics, and engineering. Specifically, a harmonic function in a bounded domain with a certain boundary value represents a particular physical quantity in equilibrium with fixed boundary data, precisely described by the Dirichlet problem of Laplace's equation. In this paper, we mainly study a (classical) solution $u \in C^2(B_R) \cap C(\overline{B}_R)$
to the following Dirichlet problem of Laplace’s equation on a disk $B_R:= \{(x,y) \in \mathbb{R}^2: x^2+y^2 <R^2  \}$,
\begin{equation} \label{maineq}
\left\{
\begin{alignedat}{2}
\Delta u &=0&& \quad \mbox{in $B_R$,}\\
u &= g &&\quad \mbox{on $\partial B_{R}$},
\end{alignedat} \right.
\end{equation}
where $R>0$ is a constant, $g \in C(\partial B_R)$ and $\Delta$ denotes Laplace's operator defined by $\Delta w = \frac{\partial^2 w}{\partial x^2}+\frac{\partial^2 w}{\partial y^2}$ for a twice continuously differentiable function $w$.
We derive various results (Theorem \ref{mainthmintro}) for a unique solution to \eqref{maineq}, including the uniform estimates, approximation, and sharp convergence rate for our constructed solution $\widetilde{u}$ to \eqref{maineq}, without using the Poisson integral formula and multi-variable integral calculus. \\
Various studies have been conducted regarding solutions to \eqref{maineq}, with a particular interest in their existence, uniqueness, regularity, and stability. The development of modern analysis, particularly the Sobolev space theory, has led to innovative ways to study the existence and uniqueness of solutions to \eqref{maineq} and their regularity. For instance, suppose there exists a continuous function $\widetilde{g}$ on $\overline{B}_R$ with a certain regularity that satisfies 
$$
\widetilde{g}|_{\partial B_R}=g \; \text{  on } \; \partial B_R.
$$ 
Let us first assume that $\widetilde{g} \in H^{1,2}(B_R)$ (Here for each $r \in [1, \infty]$, $H^{1,r}(B_R)$ denotes the space of all weakly differentiable function $w$ on $B_R$ with $w \in L^r(B_R)$ and $\|\nabla w\| \in L^r(B_R)$). Then, using the Riesz-representation theorem (or the Lax-Milgram theorem), there exists a unique $\hat{u} \in H^{1,2}_0(B_R):= \{w \in H^{1,2}(B_R):  \text{trace}(w) = 0 \text{ on } \partial B_R  \}$ such that $\Delta \hat{u} = -\text{div}(\nabla \widetilde{g})$ weakly in $B_R$, i.e.
\begin{equation} \label{weakformula}
\int_{B_R} \langle \nabla \hat{u}, \nabla \varphi \rangle dx = \int_{B_R} \langle -\nabla \widetilde{g} ,\nabla \varphi \rangle dx, \quad \text{ for all $\varphi \in C_0^{\infty}(B_R)$},
\end{equation}
where $C_0^{\infty}(B_R)$ denotes the set of all smooth functions with compact support in $B_R$ and $\text{trace}(w)$ is defined as follows (see \cite[Theorem 4.6]{EG15}): for any $w_n \in C^{1}(\overline{B}_R)$ with $\lim_{n \rightarrow \infty} w_n = w$ in $H^{1,2}(B_R)$,
$$
\text{trace}(w) := \lim_{n \rightarrow \infty} w_n \quad \text{ in $L^2(\partial B_R)$}.
$$
Defining $\bar{u}:=\hat{u}+\widetilde{g} \in H^{1,2}(B_R)$, we obtain that
\begin{equation*} \label{maineqweak}
\left\{
\begin{alignedat}{2}
\Delta \bar{u} &=0&& \quad \mbox{weakly in $B_R$,}\\
\text{trace}(\bar{u}) &= g &&\quad \mbox{on $\partial B_{R}$}.
\end{alignedat} \right.
\end{equation*}
In that case, $\bar{u}$ is called a weak solution to \eqref{maineq}. Remarkably, if\, $\widetilde{g} \in H^{1,p}(B_R) \cap C(\overline{B}_R)$ for some $p\in (2, \infty)$, then applying the classical $L^{p}$-regularity result (for instance, see \cite[Theorem 7.1]{GM12}) to \eqref{weakformula}, we are able to obtain $\hat{u} \in H^{1,p}(B_R) \cap H^{1,2}_0(B_R) \cap C(\overline{B}_R)$, and hence $\bar{u} = \hat{u} + \widetilde{g} \in H^{1,p}(B_R) \cap C(\overline{B}_R)$ satisfying that $\bar{u}(x) = g(x)$ for all $x \in \partial B_R$. Now, consider the space $H^{2,p}(B_R):=\{ w \in H^{1,p}(B_R): \partial_i \partial_j w \in L^p(B_R) \text{ for all $1 \leq i, j \leq 2$ } \}$ for some $p \in (1, \infty)$ and assume that $\widetilde{g} \in H^{2,p}(B_R) \cap C(\overline{B}_R)$. Then, it follows from \eqref{weakformula} and \cite[Theorem 9.15]{GT01} that $\hat{u} \in H^{2,p}(B_R) \cap H^{1,2}_0(B_R) \cap C(\overline{B}_R)$, and hence $\bar{u}= \hat{u} + \widetilde{g} \in H^{2,p}(B_R) \cap C(\overline{B}_R)$ fulfills
\begin{equation} \label{maineqstrong}
\left\{
\begin{alignedat}{2}
\Delta \bar{u}(x) &=0&& \qquad \mbox{for a.e. $x \in B_R$,}\\
\bar{u}(x) &= g(x) &&\qquad \mbox{for all  $x \in \partial B_{R}$}.
\end{alignedat} \right.
\end{equation}
For more recent $L^p$-regularity results for linear elliptic equations with lower order terms and general domains, we refer to \cite{S06, Kr08, KK19, K21, Kw22, L24, L24JM} and references therein. Meanwhile, to replace \lq\lq for a.e. $x \in B_R$\rq\rq \, in \eqref{maineqstrong} by  \lq\lq for all $x \in B_R$\rq\rq, we need additional regularity results. Indeed, if $\widetilde{g} \in C^{2, \alpha}(\overline{B}_R)$ for some $\alpha \in (0, 1]$, then applying the Schauder theory (\cite[Theorem 9.15]{GT01}) to \eqref{weakformula} we get $\hat{u} \in C^{2, \alpha}(\overline{B}_R) \cap H^{1,2}_0(B_R)$, and hence $\bar{u} =\hat{u} + \widetilde{g} \in C^{2, \alpha}(\overline{B}_R)$ becomes a (classical) solution to \eqref{maineq}. \\
The core we discussed so far has been solving the Dirichlet problem with homogeneous boundary conditions by extending the boundary data $g$ to $\widetilde{g}$ on the entire domain $\overline{B}_R$ and applying known regularity results.
However, in many cases, boundary data $g \in C(\partial B_R)$ may not be extended to a function that satisfies a certain regularity condition. Hence, we need a different argument to solve \eqref{maineq}. To derive the uniqueness of solutions to \eqref{maineq} is quite straightforward by using the weak maximum principle (see Theorem \ref{weakmaxi}). For the existence of a solution $u$ to \eqref{maineq} in case of $g \in C(\partial B_R)$, one can use the Poisson integral formula we mentioned in the front, which expresses the solution $u$ to \eqref{maineq} as a line integral form, as below (see \cite[Theorem 2.1.2]{jo13}):
\begin{equation} \label{poisson}
u(\mathbf{x}) := \frac{R^2-\|\mathbf{x}\|^2}{2\pi R}\int_{\partial B_R} \frac{g(\mathbf{x}')}{\|\mathbf{x}'-\mathbf{x}\|^2}ds(\bold{x}'), \qquad \bold{x} \in B_R.
\end{equation}
Direct calculations confirm that the extension of $u$ defined in \eqref{poisson} to $\overline{B}_R$ is indeed a unique solution to \eqref{maineq}. To calculate $u(\bold{x})$ precisely for each $\bold{x}=(x,y) \in B_R$, we can express the line integral \eqref{poisson} as the one-variable integral below:
\begin{align} \label{calculpoiso}
u(x,y) &= \frac{(R^2-x^2-y^2)}{2\pi} \int_{-\pi}^{\pi} \frac{g(R \cos \phi, R \sin \phi )}{R^2 +x^2+y^2 - 2R \left( x \cos \phi + y \sin \phi   \right)} d \phi, \qquad (x,y) \in B_R.
\end{align}
But calculating the exact value for the integral above may not be easy, and especially note that as $(x,y) \in B_R$ approaches $\partial B_R$, calculating \eqref{calculpoiso} may be more difficult because the absolute value of the integrand goes to infinity. A well-known alternative approach to approximate the solution $u$ defined in \eqref{poisson} is using polar coordinate representation. Precisely, it is known that (see \cite[Section 6.3]{Str08} and \cite[Section 3.3]{AU23})
\begin{equation} \label{polarepres}
u(x,y) = \frac12c_0 + \sum_{n=1}^{\infty} r^n \left( c_n \cos n \theta + d_n \sin n \theta \right), \quad (x,y) = (r \cos \theta, r \sin \theta) \in B_R,
\end{equation}
where $c_n = \frac{1}{\pi R^n}\int_{-\pi}^{\pi} g(R \cos \phi, R \sin \phi)  \cos n \phi\,d\phi$ and \,$d_n  =  \frac{1}{\pi R^n}\int_{-\pi}^{\pi} g(R \cos \phi, R \sin \phi) \sin n\phi\,d\phi$, $n \geq 0$. But calculating the exact value $\theta$ satisfying $(x,y) =(r \cos \theta, r \sin \theta)$ is not easy, and the convergence rate of \eqref{polarepres} may not be explicitly presented in the existing literature.
Now let us introduce our main results:
\begin{theorem}  \label{mainthmintro}
\begin{itemize}
\item[(i)] 
Let $g \in C(\partial B_R)$. Then, there exists $\widetilde{u} \in C(\overline{B}_R) \cap  C^\infty(B_R)$ such that $\widetilde{u}$ is a (unique) solution to \eqref{maineq}.  Moreover, $\widetilde{u}$ enjoys the following uniform estimates via the $L^1$-boundary data: for each $r \in [0, R)$, $(x,y) \in \overline{B}_r$ and $\alpha_1, \alpha_2 \in \mathbb{N} \cup \{0\}$, it holds that
\begin{align} \label{uniformderl1}
 |D^{(\alpha_1, \alpha_2)} \widetilde{u}(x,y)| &\leq  D^{(\alpha_1, \alpha_2)} \left(\frac{-1}{2\pi R}\int_{\partial B_R} |g|\,ds \right) + \frac{(\alpha_1+\alpha_2)!}{\pi (R-r)^{\alpha_1+\alpha_2+1}}  \int_{\partial B_R}|g| d s,
\end{align}
where $\overline{B}_0:=\{\bold{0}\}$.

\item[(ii)]
$\widetilde{u}$ in (i) is analytic in $B_R(\mathbf{0})$. In particular, for each $\bold{x}_0=(x_0,y_0) \in B_R$, $\kappa \in (0,\frac{1}{3})$, 
$\bold{h} =(h_1, h_2) \in \mathbb{R}^2$ with $\| \bold{h}\|\in [0, \kappa L)$ with $L:= R-\|\bold{x}_0\|$, it holds that
\begin{equation} \label{taylorexpress}
\widetilde{u}(\bold{x}_0 + \bold{h})= \sum_{k=0}^{\infty}\left( \sum_{\alpha_1+\alpha_2=k}  \frac{D^{(\alpha_1, \alpha_2)} \widetilde{u}(\bold{x}_0) }{\alpha_1 ! \alpha_2!} h_1^{\alpha_1} h_2^{\alpha_2} \right), \quad 
\end{equation}
and that
\begin{equation}  \label{analytiestim}
\left| \widetilde{u}(\bold{x}_0 + \bold{h})-\sum_{k=0}^{n-1}\left( \sum_{\alpha_1+\alpha_2=k}  \frac{D^{(\alpha_1, \alpha_2)} \widetilde{u}(\bold{x}_0) }{\alpha_1 ! \alpha_2!} h_1^{\alpha_1} h_2^{\alpha_2} \right)  \right| \leq \left(\frac{2 \kappa }{1-\kappa }\right)^n \frac{1}{\pi( L-\kappa L)}  \int_{\partial B_{L} (\bold{x}_0)}|\widetilde{u}| ds.
\end{equation}

\item[(iii)]  
Assume that $g \in C^{\alpha}(\partial B_R)$ for some $\alpha \in (0,1]$. Let $f(\theta):=g(R \cos \theta, R \sin \theta)$, $\theta \in [-\pi, \pi]$.
(Then, $f \in C^{\alpha}([-\pi, \pi])$ with $[f]_{C^{\alpha}([-\pi, \pi])} \leq R^{\alpha} [g]_{C^{\alpha}(\partial B_R)}$ by Proposition \ref{holestim}.) For each $n \geq 1$, let $\widetilde{u}_n$ be a polynomial defined by
$$
\widetilde{u}_n(x,y):=\frac{c_0}{2}+\sum_{k=1}^{n}\left(c_k\sum_{j \; \text{even}  \atop 0 \leq j \leq k}^{}(-1)^\frac{j}{2}\binom{k}{j}x^{k-j}y^{j}+d_k\sum_{j \; \text{odd}  \atop 0 \leq j \leq k}^{}(-1)^\frac{j-1}{2}\binom{k}{j}x^{k-j}y^{j}\right), \quad (x,y) \in \mathbb{R}^2,
$$
where
$$
c_k:=\frac{1}{\pi}\int_{-\pi}^{\pi}f(\theta)\cos k\theta\, d\theta, \quad d_k:=\frac{1}{\pi}\int_{-\pi}^{\pi}f(\theta)\sin k\theta\, d\theta, \quad k \geq 0.
$$
Then, $\widetilde{u}_n$ converges to $\widetilde{u}$ uniformly on $\overline{B}_R$ as $n \rightarrow \infty$. In particular, the following error estimates are fulfilled:  for each $n \geq e^{\frac{1}{\alpha}}$ 
$$
|\widetilde{u}(x,y)-\widetilde{u}_n(x,y)| \leq  2\gamma_0 (2\pi)^{\alpha}[f]_{C^\alpha([-\pi, \pi])} \left(\frac{\sqrt{x^2+y^2}}{R}\right)^{n+1} (\ln n)\left( \frac{1}{n} \right)^{\alpha},\quad \forall \text{$(x,y) \in \overline{B}_R$},
$$
where $\gamma_0>0$ is a constant as in Lemma \ref{jackson}.

\item[(iv)] Let $g$ and $f$ be defined as in (iii) and assume that $f \in C^{k, \alpha}([-\pi, \pi])$ for some $k \in \mathbb{N}$. Let $(\widetilde{u}_n)_{n \geq 1}$ be defined as in (iii). Then, the following improved error estimates are fulfilled:
 for each $n \geq e$
\begin{align*} 
|\widetilde{u}(x,y)-\widetilde{u}_n(x,y)| \leq  2\gamma_k (2\pi)^{\alpha}[f^{(k)}]_{C^\alpha([-\pi, \pi])} \left(\frac{\sqrt{x^2+y^2}}{R}\right)^{n+1} \ln n \cdot \left( \frac{1}{n} \right)^{k+\alpha} \quad \forall \text{$(x,y) \in \overline{B}_R$},   \label{mainestimer2}
\end{align*}
where $\gamma_k$ is a constant as in Lemma \ref{jackson2}. 
\end{itemize}
\end{theorem}
The proofs of Theorem \ref{mainthmintro}(i) and (ii) are presented in the ones of Theorem \ref{mainthmi} and Corollary  \ref{holomoresul}, respectively. The proof of Theorem \ref{mainthmintro}(iii) is presented in the one of Theorem \ref{limitharmo}.  The proof of Theorem \ref{mainthmintro}(iv) is discussed in Section \ref{lastsecacc}.
We here emphasize that our main results are proved independently without utilizing the Poisson integral formula \eqref{poisson} and multi-variable integral calculus, the mean value property (see \cite[Theorem 1.6]{HL11}), and the divergence theorem. \\
The motivation for our work stems from the results in \cite[Section 2.7]{Kr96} where harmonic polynomials that converge uniformly to a solution to \eqref{maineq} are constructed by using the uniform estimates and the weak maximum principle. However, the existence of the harmonic polynomial in \cite[Section 2.7]{Kr96} was proven using the isomorphic property for finite-dimensional vector spaces in linear algebra, and hence, the constructed one may not be explicit, and also, we may not know in \cite[Theorem 2.7.8]{Kr96} the exact convergence rate in which the harmonic polynomial converges uniformly to a solution to \eqref{maineq}.
On the other hand, in our main result (Theorem \ref{mainthmintro}(iii)), we can explicitly construct harmonic polynomials $(\widetilde{u}_n)_{n \geq 1}$ approximating a solution to \eqref{maineq} uniformly and also present a specific convergence rate. Remarkably, it can be seen that the smoother the boundary data $g$ and the closer the target point in $(x,y)$ to the center in the disk, we obtain the more accelerated convergence rate.
Another interesting feature in our main results is the uniform estimates for the derivative of our solution with arbitrary order in terms of the $L^1$-boundary data. 
The constant of the right-hand side in \eqref{uniformderl1} is quite simple (see Remark \ref{l1estimconsim}), and hence our estimate is quite distinguishable from the other known uniform estimates 
derived from the mean value property. \\
As a probabilistic counterpart for our results, let us consider a standard Brownian motion $(W_t)_{t \geq 0}$ in $\mathbb{R}^2$ with a filtered probability space $(\Omega, \mathcal{F}, \mathbb{P}, (\mathcal{F}_t)_{t \geq 0})$ and let $D_{R} := \inf \{ t \geq 0 : \|W_t\| \geq R  \}$. Then, it is known that (see \cite[Theore 9.2.5]{O11}) the unique solution $\widetilde{u}$ to \eqref{maineq} is represented as
\begin{equation} \label{stochasticou}
\widetilde{u}(\bold{x}) = \mathbb{E}\left[ g\left(W_{D_R} +\bold{x}]\right) \right], \quad \;\; \bold{x} \in B_R,
\end{equation}
where $\mathbb{E}$ is the expectation with respect to $(\Omega, \mathcal{F}, \mathbb{P})$. Then, applying Theorem \ref{mainthmintro} to \eqref{stochasticou} enables us to control the stochastic quantity $\mathbb{E}\left[ g\left(W_{D_R} +\bold{x}]\right) \right]$ by the $L^1$-boundary data $\|g\|_{L^1(\partial B_R)}$. Finally, we mention that Theorem \ref{mainthmintro}(ii) and particularly \eqref{analytiestim} are derived based on the estimate \eqref{uniformderl1}. There, the radius of the convergence for Taylor's series of $f$ centered at $\bold{x}_0 \in B_R$ is also explicitly given as $\frac13 (R - \|\bold{x}_0\|)$, and it looks more improved than the previous one in \cite[Theorem 10, Section 2.2]{E10} and \cite[Theorem 1.14]{HL11}, which also could be realized by our estimate \eqref{uniformderl1} (see Remark \ref{improradiu}).\\
In summary, without using the Poisson integral formula, we have independently proven the existence and regularity of the solution to the Dirichlet problem for Laplace's equations. Additionally, we have explicitly constructed a harmonic polynomial that converges to our solution $\widetilde{u}$ and demonstrated a specific convergence rate. The advantage of our method lies in deriving non-trivial high-order uniform estimates via the $L^1$-boundary data (\eqref{uniformderl1} in Theorem \ref{mainthmintro}(i)) and it allows us to achieve improved results for the constants of the right-hand side in the high-order uniform estimates in Corollary \ref{improconst}. Moreover, we obtain a refined convergence region and the error estimates within this region for the remainder of the Taylor series of our constructed solution (\eqref{analytiestim} in Theorem \ref{mainthmintro}). These have an impact on the theory of partial differential equations and numerical analysis in PDEs. \\
Our paper is organized as follows. Section \ref{notandcon} introduces the notation and conventions used throughout this paper. In Section \ref{cartentrans}, we rigorously define transformations between Cartesian and polar coordinates. Section \ref{subsecpola} discusses the Laplace operator in polar coordinates. In Section \ref{detercoef}, we study classical results on the uniform convergence of the Fourier series developed by Dunham Jackson, which plays a key role in deriving the convergence rates of our uniform approximation. In Section \ref{constharpol}, we convert functions defined in polar coordinates to Cartesian coordinates and extend them to harmonic polynomials. In Section \ref{uniforconpol}, we deal with uniform approximation by harmonic polynomials based on summation by parts formula. Then, in Section \ref{refsec}, we show that the uniform limit of the harmonic polynomial is a unique solution to \eqref{maineq} and further discuss our main uniform estimates with arbitrary order via $L^1$-boundary data. In the final section, we present the fact (Theorem \ref{mainthmintro}(iv)) that the convergence rate accelerates when the boundary data is more smoothed.

\section{Notations and conventions} \label{notandcon}
This section briefly introduces the notations and conventions we mainly use in this paper. $\Delta$ denotes the Laplace operator for  twice continuously differentiable function $f$ on $\mathbb{R}^2$. In particular, it is expressed as $\Delta f = \frac{\partial^2 f}{\partial x^2} + \frac{\partial^2 f}{\partial y^2}$.
For $R>0$ and $\mathbf{x}_0 \in \mathbb{R}^2$, we define
\begin{align*}
B_R(\mathbf{x}_0):=\{\mathbf{x} \in \mathbb{R}^2 : \|\mathbf{x}-\mathbf{x}_0 \| < R\},\\
\partial B_R(\mathbf{x}_0):=\{\mathbf{x} \in \mathbb{R}^2 : \|\mathbf{x}-\mathbf{x}_0 \| = R\},\\
\overline{B}_R(\mathbf{x}_0):=\{\mathbf{x} \in \mathbb{R}^2 : \|\mathbf{x}-\mathbf{x}_0 \| \leq R\},
\end{align*}
where $\|\cdot \|$ denotes the Euclidean norm. Specifically, we write $B_R:=B_R(\mathbf{0})$, $\partial B_R:=\partial B_R(\mathbf{0})$, and $\overline{B}_R:=\overline{B}_R(\mathbf{0})$, where $\bold{0}=(0,0)$ denotes the origin in $\mathbb{R}^2$. We define $\overline{B}_0:=\{\mathbf{0}\}$.
For $\alpha_1,\alpha_2 \in \mathbb{N} \cup \{0\}$, $U \subseteq \mathbb{R}^2$, and $\alpha_1+\alpha_2$-times continuously differentiable function $f: U \rightarrow \mathbb{R}$, we define
$$
D^{(\alpha_1,\alpha_2)}f := \frac{\partial^{\alpha_1}}{\partial x^{\alpha_1}}\frac{\partial^{\alpha_2}}{\partial y^{\alpha_2}}f.
$$
Here, note that the order of differentiation can be interchanged by Clairaut's theorem.
From now on, let $U$ be a bounded open subset of $\mathbb{R}^d$ with $d \geq 1$. Define a function space $C(\overline{U})$ as
\begin{align*}
C(\overline{U}):=\{f:\overline{U} \rightarrow \mathbb{R} : f \text{ is continuous}\}
\end{align*}
with the norm
$$
\|f\|_{C(\overline{U})} := \max_{x \in \overline{U}}|f(x)|.
$$
For $k \in \mathbb{N}$, we define functions spaces $C^k(U)$ and $C^k(\overline{U})$ as
\begin{align*}
&C^k(U):=\{f:U \rightarrow \mathbb{R} : f \text{ is $k$-times continuously differentiable}\}, \\
&C^k(\overline{U}):=\{f \in C^k(U): D^{(\alpha_1,\alpha_2)}f \text{ is uniformly continuous on $U$ for all $\alpha_1,\alpha_2 \in \mathbb{N}\cup\{0\}$ with $\alpha_1+\alpha_2 \leq k$} \}.
\end{align*}
For each $f \in C^k(\overline{U})$ with $k \in \mathbb{N}$, define 
$$
\|f\|_{C^k(\overline{U})}:=\sum_{\alpha_1+\alpha_2 \leq k} \|D^{(\alpha_1, \alpha_2)} f \|_{C(\overline{U})}.
$$
Then $f$ and every $k$-th partial derivative of $f$ on $U$ continuously extends to $\overline{U}$.
Denote by $C_0^{\infty}(U)$ the set of all smooth functions with compact support in $U$. We define a function space $C^\alpha(\overline{U})$ as
$$
C^\alpha(\overline{U}):=\left\{ f \in C(\overline{U}) : \sup_{x,y \in U}\frac{|f(x)-f(y)|}{\|x-y\|^\alpha} < \infty \right \}
$$
with the norm
$$
\|f\|_{C^\alpha(\overline{U})} := \|f\|_{C(\overline{U})} + [f]_{C^\alpha(\overline{U})},
$$
where H\"older seminorm $[\cdot]_{C^\alpha(\overline{U})}$ is defined by
$$
[f]_{C^\alpha(\overline{U})} := \sup_{x,y \in \overline{U}}\frac{|f(x)-f(y)|}{\|x-y\|^\alpha}.
$$
For $k \in \mathbb{N}$, $0< \alpha \leq 1$, we define 
$$
C^{k,\alpha}(\overline{U}):=\{ f \in C^k(\overline{U}) : D^{(\alpha_1,\alpha_2)}f \in C^\alpha(\overline{U}) \text{ for all $\alpha_1,\alpha_2 \in \mathbb{N}\cup\{0\}$ with $\alpha_1+\alpha_2 = k$} \}.
$$

\section{Solving the problem in the polar coordinates} 
\subsection{Coordinate transformations} \label{cartentrans}
We will rigorously solve the equation \eqref{maineq} defined in the Cartesian coordinate system by transforming its coordinate to the polar coordinate system, and finally, we will transform it back to the Cartesian coordinate system to obtain a solution to \eqref{maineq}. As a result, the open disk centered at the origin, excluding a line segment in the Cartesian coordinate system, can be transformed into an open rectangle in the polar coordinate system. The main advantage is that one may easily handle the boundary data $g$ as a one-variable function. Although roughly solving \eqref{maineq} using transformation via trigonometric functions is well-known in many PDE textbooks, one has to solve \eqref{maineq} rigorously 
because there is a half-line section (we will write it as $S$) where we cannot find a good one-to-one correspondence between the two coordinate systems. Our main task will be filling up the half-line segment and making rigorous arguments.
\centerline{}
We define coordinate functions $\mathbf{x}:[0,\infty) \times \mathbb{R} \rightarrow \mathbb{R}^2$ and $\bold{y}:[0,\infty) \times \mathbb{R} \rightarrow \mathbb{R}^2$ as 
$$
\bold{x}(r,\theta):=r\cos\theta, \quad  \bold{y}(r,\theta):=r\sin\theta, \quad (r, \theta) \in [0, \infty) \times \mathbb{R}.
$$
Then the functions $\mathbf{x}$ and $\mathbf{y}$ are well-defined and smooth. Define our half line segment $S$ by 
$$
S:=\{(x,y) \in \mathbb{R}:-\infty < x \leq 0, \; y=0\}. 
$$
Then, the function $(\mathbf{x},\mathbf{y})$ is bijective from $(0, \infty) \times (-\pi,\pi)$ to $\mathbb{R}^2 \setminus S$, and hence it has an inverse. Define functions $\mathbb{R}^2 \rightarrow [0,R]$ and $\Theta:\mathbb{R}^2 \setminus S \rightarrow (-\pi,\pi)$ given by
$$
\bold{r}(x,y)=\sqrt{x^2+y^2}, \;\;(x,y) \in \mathbb{R}^2, \qquad 
\Theta(x, y)=\begin{cases}
\arctan \frac{y}{x} & \text{ if } x>0,  \\
-\arctan \frac{x}{y}+\frac{\pi}{2} & \text{ if } x\leq 0, \ y>0, \\
-\arctan \frac{x}{y}-\frac{\pi}{2} & \text{ if } x\leq 0, \ y<0.  
\end{cases}
$$
Then we observe that
$$
(\bold{x},\bold{y}) \circ (\bold{r}, \Theta) = id, \quad \text{ on \,$\mathbb{R}^2 \setminus S$}, 
$$
i.e.
\begin{equation} \label{inv1}
\bold{x}(\bold{r}(x,y), \Theta(x,y)) = x, \quad  \; \bold{y}(\bold{r}(x,y), \Theta(x,y)) =y, \quad \text{for all $(x,y) \in \mathbb{R}^2 \setminus S$},
\end{equation}
and that
$$
(\bold{r}, \Theta) \circ (\bold{x},\bold{y})  = id, \quad  \; \text{ on  $(0, \infty) \times (-\pi, \pi)$},
$$
i.e.
\begin{equation} \label{inv2}
\bold{r}\Big(\bold{x}(r, \theta), \bold{y}(r, \theta)\Big)=r, \;\; \;  \Theta\Big(\bold{x}(r, \theta), \bold{y}(r, \theta)\Big)=\theta, \quad \text{ for all $(r,\theta) \in (0, \infty) \times (-\pi, \pi)$}.
\end{equation}

\subsection{Solutions in the polar coordinate}  \label{subsecpola}
Given $u \in C^2(B_R\setminus S) \cap C(\overline{B}_R\setminus S )$, define 
\begin{equation} \label{defnv}
v(r,\theta):=u(\bold{x}(r,\theta),\bold{y}(r,\theta)), \quad (r,\theta) \in (0,R] \times (-\pi, \pi).
\end{equation}
For our boundary data $g \in C(\partial B_R)$ in \eqref{maineq}, we define a function $f:\mathbb{R} \rightarrow \mathbb{R}$ given by
\begin{equation} \label{ffunc}
f(\theta):=g(\bold{x}(R,\theta),\bold{y}(R,\theta))=g(R \cos \theta, R \sin \theta), \qquad \theta \in \mathbb{R}.
\end{equation}
Then, it follows from \eqref{inv1} that
\begin{align*}
u(x,y)&=v(\bold{r}(x,y), \Theta(x,y)), \quad \text{ for all } (x,y) \in \overline{B}_R \setminus S, \\
g(x,y)&=f(\Theta(x,y)), \quad \text{ for all } (x,y) \in \partial B_R \setminus \{(-R,0)\}.
\end{align*}
Now, we apply the chain rule to the function $v$ defined in \eqref{defnv}. \, For each $(r,\theta) \in (0,R) \times (-\pi, \pi)$, we have
\begin{align*}
\frac{\partial v}{\partial \theta}(r,\theta)&=\frac{\partial u}{\partial x}(\bold{x}(r,\theta),\bold{y}(r,\theta))(-r\sin\theta)+\frac{\partial u}{\partial y}(\bold{x}(r,\theta),\bold{y}(r,\theta))(r\cos\theta),
\\
\frac{\partial^2 v}{{\partial \theta}^2}(r,\theta)&=\frac{\partial^2 u}{{\partial x}^2}(\bold{x}(r,\theta),\bold{y}(r,\theta))(-r\sin\theta)^2+\frac{\partial u}{\partial x}(\bold{x}(r,\theta),\bold{y}(r,\theta))(-r\cos\theta) \\
&\qquad +\frac{\partial^2 u}{{\partial y}^2}(\bold{x}(r,\theta),\bold{y}(r,\theta))(r\cos\theta)^2+\frac{\partial u}{\partial y}(\bold{x}(r,\theta),\bold{y}(r,\theta))(-r\sin\theta)
\\
&\qquad \qquad  + \frac{\partial^2 u}{\partial y \partial x}(\bold{x}(r, \theta), \bold{y}(r, \theta)) (-r^2 \sin \theta \cos \theta), \\
\frac{\partial v}{\partial r}(r,\theta)&=\frac{\partial u}{\partial x}(\bold{x}(r,\theta),\bold{y}(r,\theta))(\cos\theta)+\frac{\partial u}{\partial y}(\bold{x}(r,\theta),\bold{y}(r,\theta))(\sin\theta),
\\
\frac{\partial^2 v}{{\partial r}^2}(r,\theta)&=\frac{\partial^2 u}{{\partial x}^2}(\bold{x}(r,\theta),\bold{y}(r,\theta))(\cos\theta)^2+\frac{\partial^2 u}{{\partial y}^2}(\bold{x}(r,\theta),\bold{y}(r,\theta))(\sin\theta)^2\\
& \qquad \qquad  + \frac{\partial^2 u}{\partial y \partial x}(\bold{x}(r, \theta), \bold{y}(r, \theta)) (\sin \theta \cos \theta).
\end{align*}
Therefore, we get
\begin{align*}
\frac{1}{r^2}\frac{\partial^2 v}{{\partial\theta}^2}(r,\theta)+\frac{\partial^2 v}{\partial r^2}(r, \theta) + \frac{1}{r} \frac{\partial v}{\partial r}(r,\theta)=\Delta u(\bold{x}(r,\theta),\bold{y}(r,\theta)), \quad \forall (r,\theta) \in (0,R) \times (-\pi, \pi).
\end{align*}
Substituting \eqref{inv1} to this equation, we obtain for any  $(x,y) \in B_R(\mathbf{0}) \setminus S$
\begin{align} \label{impoiden}
\Delta u(x,y)&=\frac{1}{\bold{r}(x,y)^2} \frac{\partial^2 v}{\partial \theta^2} (\bold{r}(x,y), \Theta(x,y)) + \frac{\partial^2 v}{\partial r^2}(\bold{r}(x,y), \Theta(x,y)) + \frac{1}{\bold{r}(x,y)} \frac{\partial v}{\partial r}(\bold{r}(x,y), \Theta(x,y)).
\end{align}
Therefore, the original problem \eqref{maineq} is now transformed into the following problem:
\begin{equation} \label{polareq0}
\left\{
\begin{alignedat}{2}
\frac{1}{r^2}\frac{\partial^2 v}{{\partial \theta}^2}(r, \theta)+\frac{\partial^2 v}{{\partial r}^2}(r, \theta)+\frac{1}{r}\frac{\partial v}{\partial r}(r, \theta) &=0&& \quad \mbox{if $(r, \theta) \in (0,R) \times (-\pi, \pi)$}\\
v(R,\theta) &= f(\theta) &&\quad \mbox{if \;\;\;$\theta \in (-\pi, \pi)$}.
\end{alignedat} \right.
\end{equation} 
\noindent
Now we aim to find a solution to \eqref{polareq0}. To do it, let us postulate that $\hat{v}(r,\theta)=L(r)\Theta(\theta)$ on $[0,\infty) \times \mathbb{R}$ where $\hat{v}$ is bounded on $[0, R]\times \mathbb{R}$ and $\Theta$ is a periodic function with $2\pi$-period and assume that
\begin{equation} \label{polardiff}
\frac{1}{r^2}\frac{\partial^2 \hat{v}}{{\partial \theta}^2}+\frac{\partial^2 \hat{v}}{{\partial r}^2}+\frac{1}{r}\frac{\partial \hat{v}}{\partial r} =0  \qquad \mbox{in $(0,\infty) \times \mathbb{R}$},
\end{equation}
which is equivalent to the fact that
$$
\frac{1}{r^2}L(r)\Theta''(\theta)+L''(r)\Theta(\theta)+\frac{1}{r}L'(r)\Theta(\theta)=0, \quad \text{ for any $(r, \theta) \in (0, \infty) \times \mathbb{R}$},
$$
and hence
$$
\frac{r^2 L''(r)+rL'(r)}{L(r)}=-\frac{\Theta''(\theta)}{\Theta(\theta)}.
$$
In the above, the right-hand side is independent of $r$, and the left-hand side is independent of $\theta$, so that they are equal to some constant, let us call it $\lambda$. Thus, we obtain the following two equations:
\begin{equation*}
r^2 L''(r)+rL'(r)-\lambda L(r)=0, \quad r \in (0, \infty), \qquad \Theta''(\theta)+\lambda \Theta(\theta)=0, \quad \theta \in \mathbb{R}.
\end{equation*}
Now consider three cases.
\begin{itemize}
\item[(a)] If $\lambda<0$, then there exists $\beta>0$ so that $\Theta''(\theta)-\beta^2\Theta(\theta)=0$, for all $\theta \in (0, \infty)$. Thus, we can find the general solution as
$$
\Theta(\theta)=\bar{c}_1 e^{\beta\theta}+\bar{c}_2 e^{-\beta\theta},  \quad \theta \in \mathbb{R},
$$
where  $\bar{c}_1,\bar{c}_2 \in \mathbb{R}$ are constants.
Since $\Theta$ is a periodic function with the $2\pi$-period, we get $\bar{c}_1=\bar{c}_2=0$, so that $\hat{v}=0$.
\item[(b)]
If $\lambda=0$, then $\Theta''(\theta)=0$ for all $\theta \in \mathbb{R}$ and $r^2L''(r)+rL'(r)=0$ for all $r \in (0, \infty)$. We can find the solution as
$$
L(r)=\bar{d}_1+\bar{d}_2\ln r, \quad \Theta(\theta)=\bar{c}_1+\bar{c}_2\theta, \quad \text{ $(r, \theta) \in (0, \infty) \times \mathbb{R}$},
$$
where $\bar{c}_1,\bar{c}_2,\bar{d}_1,\bar{d}_2 \in \mathbb{R}$ are constants. Since $\Theta$ is a periodic function and $\hat{v}$ is a bounded on $[0,R]$, we get $\bar{c}_2=\bar{d}_2=0$, and hence $v$ is a constant function.
\item[(c)] If $\lambda>0$, then there exists $\beta>0$ such that $r^2L''(r)+rL'(r)-\beta^2L(r)=0$ for all $r \in (0, \infty)$ and $\Theta''(\theta)+\beta^2\Theta(\theta)=0$ for all $\theta \in \mathbb{R}$. We hence obtain that
$$
L(r)=\bar{d}_1r^\beta+\bar{d}_2r^{-\beta}, \quad \Theta(\theta)=\bar{c}_1\sin\beta\theta+\bar{c}_2\cos\beta\theta, \quad \text{ $(r, \theta) \in (0, \infty) \times \mathbb{R}$},
$$
where $\bar{c}_1,\bar{c}_2,\bar{d}_1,\bar{d}_2 \in \mathbb{R}$ are constants. Since $v$ is bounded and $\Theta$ is a periodic function with $2\pi$-period, it satisfies the assumption only when $\bar{d}_2=0$ and $\beta$ is a positive integer. Thus, 
$$
L(r)= \bar{d}_1 r^n \; \text{ and }\; \Theta(\theta)= \bar{c}_1 \sin n \theta + \bar{c}_2 \cos n \theta, \quad (r, \theta) \in (0, \infty) \times \mathbb{R}.
$$
\end{itemize}
\text{}
Therefore, we can see that
$$
\hat{v}(r,\theta)=r^n(\bar{c}_1\sin n\theta+\bar{c}_2\cos n\theta), \quad \text{ $(r, \theta) \in [0, \infty) \times \mathbb{R}$}
$$
fulfills \eqref{polardiff}. Note that $\hat{v}$ is not the only function satisfying \eqref{polardiff}, and we can obtain infinitely many functions as linear combinations of them. 
Precisely, for each $n \in \mathbb{N}$ we define a function $\hat{v}_n: [0,\infty) \times \mathbb{R} \rightarrow \mathbb{R}$ given by
\begin{equation} \label{defnvn}
\hat{v}_n(r,\theta)=\frac{\hat{c}_0}{2}+\sum_{k=1}^n r^k(\hat{c}_k\cos k\theta+\hat{d}_k\sin k\theta), \quad (r,\theta) \in [0,\infty) \times \mathbb{R},
\end{equation}
where $\hat{c}_k$ and $\hat{d}_k$ are (indetermined) constants for $k=0,1,2,...,n$. Then, \eqref{polardiff} is fulfilled where $\hat{v}$ is replaced by $\hat{v}_n$.

\subsection{Fourier series and Jackson's theorem} \label{detercoef}

Since $\hat{v}_n$ in \eqref{defnvn}  looks similar to $n$-th partial sum of Fourier series, we expect to get $v$ as a limit of $\hat{v}_n$ defined in \eqref{defnvn}
which fulfills
$$
v(R,\theta)=f(\theta), \quad \theta \in [-\pi,\pi],
$$
where $f$ is defined as in \eqref{ffunc}, by choosing suitable coefficients $c_k$ and $d_k$.

\noindent
Now we introduce a classical result developed by Dunham Jackson, which guarantees that the Fourier series of a regular continuous function converges uniformly with a certain convergence rate.
\begin{lemma}{\cite[page 22, Corollary II]{jack94}} \label{jackson}
Let $f$ be of $C([-\pi,\pi])$ and be periodic with $2\pi$-period.
Define
\begin{equation} \label{fouridefn}
S_n(f)(\theta) := \frac{\tilde{c}_0}{2}+\sum_{k=1}^{n} (\tilde{c}_k\cos k\theta+\tilde{d}_k \sin k\theta), \qquad \theta \in \mathbb{R}
\end{equation}
and
$$
\tilde{c}_k:=\frac{1}{\pi}\int_{-\pi}^{\pi}f(\theta)\cos k\theta d\theta, \;\;\quad \tilde{d}_k:=\frac{1}{\pi}\int_{-\pi}^{\pi}f(\theta)\sin k\theta d\theta, \quad k \geq 0.
$$
Then, for any $\theta \in [-\pi, \pi]$ and $n \in \mathbb{N}$, it holds that
$$
|f(\theta)-S_n(f)(\theta)| \leq \gamma_0 \ln n \cdot \omega_f\left(\frac{2\pi}{n}\right),
$$
where $\gamma_0$ is a universal constant and
$\omega_f$ is defined as
$$
\omega_f(\delta) := \sup_{\substack{\theta_1,\theta_2 \in [-\pi,\pi] \\ |\theta_1-\theta_2|\leq \delta}}|f(\theta_1)-f(\theta_2)|, \quad \delta>0.
$$
\end{lemma}

\begin{corollary} 
Let $f$ be of $C^{\alpha}([-\pi,\pi])$ for some $\alpha \in (0,1]$ and be periodic with period $2\pi$. Then, 
$$
|f(\theta)-S_n(f)(\theta)| \leq  (2\pi)^{\alpha} \gamma_0 [f]_{C^{\alpha}([-\pi, \pi])}\cdot  \left( \frac{1}{n} \right)^{\alpha} \ln n,
$$
where $S_n(f)$ is defined as in \eqref{fouridefn} and $\gamma_0$ is a constant as in Lemma \ref{jackson}. Thus, $S_n(f)$ converges to $f$ uniformly on $[-\pi, \pi]$ as $n \rightarrow \infty$.
\end{corollary}
\begin{proof}
By the definition of the H\"older continuity,
$$
|f(\theta_1)-f(\theta_2)| \leq [f]_{C^\alpha([-\pi, \pi])}|\theta_1-\theta_2|^\alpha, \quad \theta_1,\theta_2 \in [-\pi,\pi],
$$
so that $w_f(\delta) \leq [f]_{C^{\alpha}([-\pi, \pi])} \delta^{\alpha}$ for all $\delta>0$.
Thus, the assertion follows from Lemma \ref{jackson}.
\end{proof}
\centerline{}
\noindent
Let us note that in the original problem \eqref{maineq}, $f$ defined in \eqref{ffunc} does not explicitly appear. So it would be better to give the H\"{o}lder continuity assumption to $g$ rather than $f$. In the following proposition, we will investigate the relation between the H\"older continuity of $g$ and $f$.
\begin{proposition} \label{holestim}
Let $g \in C^{\alpha}(\partial B_R)$ for some $\alpha \in (0, 1]$ and $f$ be defined as in \eqref{ffunc}. Then, $f \in C^\alpha([-\pi,\pi])$ and
\begin{equation*} \label{Calphanorm}
[f]_{C^\alpha([-\pi,\pi])} \leq R^{\alpha}[g]_{C^\alpha(\partial B_R)}.
\end{equation*}
\end{proposition}

\begin{proof}
For each $\theta_1,\theta_2 \in [-\pi,\pi]$, it holds that
\begin{align*}
|f(\theta_1)-f(\theta_2)| &= \Big|g(R\cos\theta_1,R\sin\theta_1)-g(R\cos\theta_2,R\sin\theta_2)\Big| \\
&\leq [g]_{C^\alpha(\partial B_R)}\left(\sqrt{(R\cos\theta_1-R\cos\theta_2)^2+(R\sin\theta_1-R\sin\theta_2)^2}\right)^\alpha \\
&= [g]_{C^\alpha( \partial B_R)}\left(\sqrt{2R^2-2R^2\cos(\theta_1-\theta_2)}\right)^\alpha \\
&=(2R)^{\alpha}[g]_{C^\alpha( \partial B_R)}\left(\sin\frac{|\theta_1-\theta_2|}{2}\right)^\alpha \leq  R^{\alpha} [g]_{C^\alpha(\partial B_R)} |\theta_1-\theta_2|^\alpha.
\end{align*}
Hence, the assertion follows.
\end{proof}
\centerline{}
All the results in Sections \ref{subsecpola}, \ref{detercoef} so far are summarized as follows.
\begin{proposition} \label{uniconth}
Let $g \in C^\alpha(\partial B_R)$ for some $\alpha \in (0,1]$ and $f$ be defined as in \eqref{ffunc}. (Then, $f \in C^{\alpha}([-\pi, \pi])$ with $[f]_{C^{\alpha}([-\pi, \pi])} \leq R^{\alpha} [g]_{C^{\alpha}(\partial B_R)}$ by Proposition \ref{holestim}.)
For each $k \geq 0$, define
\begin{equation} \label{ckdkdefn}
c_k:=\frac{1}{R^k\pi}\int_{-\pi}^{\pi} f(\theta)\cos k\theta d\theta, \;\quad d_k:=\frac{1}{R^k\pi}\int_{-\pi}^{\pi} f(\theta)\sin k\theta d\theta.
\end{equation}
Let
\begin{equation} \label{fourierdefn}
v_n(r,\theta)=\frac{c_0}{2}+\sum_{k=1}^n r^k(c_k\cos k\theta+d_k\sin k\theta), \quad (r,\theta) \in [0,\infty) \times \mathbb{R}.
\end{equation}
Then, 
\begin{equation} \label{polardiff2}
\frac{1}{r^2}\frac{\partial^2 v_n}{{\partial \theta}^2}(r, \theta)+\frac{\partial^2 v_n}{{\partial r}^2}(r, \theta)+\frac{1}{r}\frac{\partial v_n}{\partial r}(r, \theta) =0  \qquad \mbox{ for all  $(r, \theta) \in (0,\infty) \times \mathbb{R}$}.
\end{equation}
Moreover, it holds that
\begin{align*}
|f(\theta)-v_n(R, \theta)| &\leq (2\pi)^{\alpha} \gamma_0 [f]_{C^{\alpha}([-\pi, \pi])} \left( \frac1n \right)^{\alpha} \ln n.
\end{align*}
In particular, $v_n(R,\cdot)$ converges to $f$ uniformly on $[-\pi, \pi]$ as $n \rightarrow \infty$.
\end{proposition}
\centerline{}

\section{Constructing a solution in the Cartesian coordinates} 
\subsection{Constructing harmonic polynomials} \label{constharpol}
Let $g \in C^\alpha(\partial B_R)$ for some $\alpha \in (0, 1]$ and $f$ be a function on $\mathbb{R}$ defined in \eqref{ffunc}. For each $n \in \mathbb{N}$, let $v_n$ be a function on $[0,\infty) \times \mathbb{R}$
defined in \eqref{fourierdefn}. For each $n \in \mathbb{N}$, we define a function $u_n$ on $\mathbb{R}^2 \setminus S$ by 
\begin{align}
u_n(x,y)&:=v_n(\mathbf{r}(x,y),\Theta(x,y)) \nonumber \\
&=\frac{c_0}{2}+\sum_{k=1}^{n} \mathbf{r}(x,y)^k(c_k\cos k\Theta(x,y)+d_k\sin k\Theta(x,y)), \quad (x,y) \in \mathbb{R}^2 \setminus S, \label{undefnser} 
\end{align}
where $c_k$ and $d_k$ are defined as in \eqref{ckdkdefn}. Note that for each $n \in \mathbb{N}$, $u_n \in C^2(\mathbb{R}^2 \setminus S)$ and it follows from \eqref{inv2} that
\begin{align} \label{harmoe}
u_n(\mathbf{x}(R,\theta),\mathbf{y}(R,\theta)) = v_n(R, \theta), \quad \text{  for all }\theta \in \mathbb{R}.
\end{align}
Moreover, by the calculation for \eqref{impoiden} and \eqref{polardiff2}, we obtain that
\begin{align} 
\Delta u_n(x,y)&=\frac{1}{\bold{r}(x,y)^2} \frac{\partial^2 v_n}{\partial \theta^2} (\bold{r}(x,y), \Theta(x,y)) \nonumber \\
&\quad + \frac{\partial^2 v_n}{\partial r^2}(\bold{r}(x,y), \Theta(x,y)) + \frac{1}{\bold{r}(x,y)} \frac{\partial v_n}{\partial r}(\bold{r}(x,y), \Theta(x,y)) =0 \qquad \forall (x,y)\in \mathbb{R}^2 \setminus S. \label{harmonicun}
\end{align} 
We aim to make a continuous extension of $u_n$ on $\mathbb{R}^2$.
 First, let us derive the expression of $\cos \Theta(x,y)$ and $\sin \Theta(x,y)$. For each $(x,y) \in \mathbb{R}^2 \setminus S$, we obtain that
\begin{align}
\cos \Theta(x,y)
&=\frac{x}{\sqrt{x^2+y^2}}, \label{coskthdouble}
\end{align}
and that
\begin{align}
\sin \Theta(x,y)
&=\frac{y}{\sqrt{x^2+y^2}}. \label{sinkthdouble}
\end{align}
We will use these expressions to extend $\sin k\Theta(x,y)$ and $\cos k\Theta(x,y)$ on $\mathbb{R}^2 \setminus \{  \bold{0}\}$ for each $k \geq 1$.
The following is the $k$-double angle formula for trigonometric functions. We left the state and its proof for the reader's accessibility.
\begin{lemma} \label{lemdemoiv}
For any $\theta \in \mathbb{R}$ and $k \in \mathbb{N}$ it holds that
\begin{align*}
\cos k\theta=\sum_{j \; \text{even}  \atop 0 \leq j \leq k}^{}(-1)^\frac{j}{2}\binom{k}{j}\cos^{k-j}\theta\sin^{j}\theta, \\
\sin k\theta=\sum_{j \; \text{odd}  \atop 0 \leq j \leq k}^{}(-1)^\frac{j-1}{2}\binom{k}{j}\cos^{k-j}\theta\sin^{j}\theta.
\end{align*}
\end{lemma}
\begin{proof}
By de Moivre's formula, 
\begin{align*}
\cos k\theta+i\sin k\theta&=(\cos\theta+i\sin\theta)^k \\
&=\sum_{j=0}^k \binom{k}{j}\cos^{k-j}\theta(i\sin\theta)^j \\
&=\sum_{j \; \text{even}  \atop 0 \leq j \leq k}^{}(-1)^\frac{j}{2}\binom{k}{j}\cos^{k-j}\theta\sin^{j}\theta+i\sum_{j \; \text{odd}  \atop 0 \leq j \leq k}^{}(-1)^\frac{j-1}{2}\binom{k}{j}\cos^{k-j}\theta\sin^{j}\theta,
\end{align*}
where $i$ is the imaginary unit. Since $\sin\theta$, $\cos\theta$, $\sin k\theta$ and $\cos k\theta$ are real numbers,
\begin{align*}
\cos k\theta=\sum_{j \; \text{even}  \atop 0 \leq j \leq k}^{}(-1)^\frac{j}{2}\binom{k}{j}\cos^{k-j}\theta\sin^{j}\theta, \qquad \sin k\theta=\sum_{j \; \text{odd}  \atop 0 \leq j \leq k}^{}(-1)^\frac{j-1}{2}\binom{k}{j}\cos^{k-j}\theta\sin^{j}\theta.
\end{align*}
\end{proof}

\noindent
Consequently, by \eqref{coskthdouble}, \eqref{sinkthdouble} and Lemma \ref{lemdemoiv}, for each $k, n \in \mathbb{N}$ and $(x,y) \in \mathbb{R}^2 \setminus S$ we obtain that
\begin{align}
\cos k\Theta(x,y)&=\sum_{j \; \text{even}  \atop 0 \leq j \leq k}^{}(-1)^\frac{j}{2}\binom{k}{j}\left(\frac{x}{\sqrt{x^2+y^2}}\right)^{k-j}\left(\frac{y}{\sqrt{x^2+y^2}}\right)^{j}  \nonumber \\
&=\frac{1}{\sqrt{x^2+y^2}^k}\sum_{j \; \text{even}  \atop 0 \leq j \leq k}^{}(-1)^\frac{j}{2}\binom{k}{j}x^{k-j}y^{j}, \label{cosexpre}
\end{align}
\begin{align}
\sin k\Theta(x,y)&=\sum_{j \; \text{odd}  \atop 0 \leq j \leq k}^{}(-1)^\frac{j-1}{2}\binom{k}{j}\left(\frac{x}{\sqrt{x^2+y^2}}\right)^{k-j}\left(\frac{y}{\sqrt{x^2+y^2}}\right)^{j} \nonumber \\
&=\frac{1}{\sqrt{x^2+y^2}^k}\sum_{j \; \text{odd}  \atop 0 \leq j \leq k}^{}(-1)^\frac{j-1}{2}\binom{k}{j}x^{k-j}y^{j}. \label{sinexpre}
\end{align}
Moreover, by \eqref{undefnser}, \eqref{cosexpre} and \eqref{sinexpre},
\begin{align}
u_n(x,y)&=\frac{c_0}{2}+\sum_{k=1}^{n} \sqrt{(x^2+y^2)}^k(c_k\cos k\Theta(x,y)+d_k\sin k\Theta(x,y)) \nonumber \\
&=\frac{c_0}{2}+\sum_{k=1}^{n}\left(c_k\sum_{j \; \text{even}  \atop 0 \leq j \leq k}^{}(-1)^\frac{j}{2}\binom{k}{j}x^{k-j}y^{j}+d_k\sum_{j \; \text{odd}  \atop 0 \leq j \leq k}^{}(-1)^\frac{j-1}{2}\binom{k}{j}x^{k-j}y^{j}\right), \label{equationnumb}
\end{align}
where $c_k$ and $d_k$ are defined as in \eqref{ckdkdefn}.
For each $k \in \mathbb{N}$, we define functions $Sin_k:\mathbb{R}^2 \setminus \{\mathbf{0}\} \rightarrow \mathbb{R}$ and $Cos_k:\mathbb{R}^2 \setminus \{\mathbf{0}\} \rightarrow \mathbb{R}$ given by
\begin{align}
Sin_k(x,y)&:=\frac{1}{\sqrt{x^2+y^2}^k}\sum_{j \; \text{odd}  \atop 0 \leq j \leq k}^{}(-1)^\frac{j-1}{2}\binom{k}{j}x^{k-j}y^{j}, \label{expsincos0} \\
Cos_k(x,y)&:=\frac{1}{\sqrt{x^2+y^2}^k}\sum_{j \; \text{even}  \atop 0 \leq j \leq k}^{}(-1)^\frac{j}{2}\binom{k}{j}x^{k-j}y^{j}, \quad (x,y) \in \mathbb{R}^2 \setminus \{\bold{0} \}. \label{expsincos}
\end{align}
Then, by \eqref{cosexpre}, \eqref{sinexpre}, \eqref{expsincos0}, \eqref{expsincos}, $Sin_k$ and $Cos_k$ are continuous extensions of $\sin k \Theta$ and $\cos k \Theta$ from $\mathbb{R}^2 \setminus S$ to 
 $\mathbb{R}^2 \setminus \{\bold{0}\}$, respectively.
 Next, for each $n \in\mathbb{N}$, we define the polynomial $\widetilde{u}_n: \mathbb{R}^2 \rightarrow \mathbb{R}$ given by
\begin{equation} \label{tildeun}
\widetilde{u}_n(x,y):=\frac{c_0}{2}+\sum_{k=1}^{n}\left(c_k\sum_{j \; \text{even}  \atop 0 \leq j \leq k}^{}(-1)^\frac{j}{2}\binom{k}{j}x^{k-j}y^{j}+d_k\sum_{j \; \text{odd}  \atop 0 \leq j \leq k}^{}(-1)^\frac{j-1}{2}\binom{k}{j}x^{k-j}y^{j}\right), \quad (x,y) \in \mathbb{R}^2.
\end{equation}
Thus, \eqref{expsincos0} and \eqref{expsincos}, for each $n \in \mathbb{N}$\, $\widetilde{u}_n$ is a polynomial satisfying
\begin{align*}
\widetilde{u}_n(x,y)&=\frac{c_0}{2}+\sum_{k=1}^{n} \sqrt{x^2+y^2}^k\Big(c_k Cos_k(x,y)+d_k Sin_k(x,y)\Big)\;\; && \text{for all $(x,y) \in \mathbb{R}^2\setminus\{\mathbf{0}\}$}.
\end{align*}
Moreover, by \eqref{equationnumb} $\widetilde{u}_n$ is a polynomial extension of $u_n$, i.e.
\begin{equation} \label{unuequalit}
\widetilde{u}_n(x,y)=u_n(x,y), \quad \text{ for all $(x,y) \in \overline{B}_R(\mathbf{0})\setminus S$.}
\end{equation}
\centerline{}
\begin{theorem} \label{polyfinal}
Let $g \in C^\alpha(\partial B_R)$ for some $\alpha \in (0, 1]$ and $f$ be a function on $\mathbb{R}$ defined in \eqref{ffunc}.  (Then, $f \in C^{\alpha}([-\pi, \pi])$ with $[f]_{C^{\alpha}[-\pi, \pi]} \leq R^{\alpha} [g]_{C^{\alpha}(\partial B_R)}$ by Proposition \ref{holestim}.)
Then, the following hold:
\begin{itemize}
\item[(i)]
Let $\widetilde{u}_n$ be a polynomial defined in \eqref{tildeun}. For each $n \geq1$, it holds that $\Delta \widetilde{u}_n=0$ in $\mathbb{R}^2$. Moreover, 
$\widetilde{u}_n$ converges to $g$ uniformly on $\partial B_R$ as $n \rightarrow \infty$.
\item[(ii)]
Let $F$ be a continuous extension of $f \circ \Theta$ 
from $\overline{B}_R \setminus S$ to $\overline{B}_{R}(\bold{0}) \setminus \{\bold{0}\}$, where $f$ is the function defined in \eqref{ffunc}. For each $n \geq 1$, define
\begin{equation} \label{Phidefn}
\Phi_n(x,y):=\frac{c_0}{2}+\sum_{k=1}^{n} R^k\Big(c_k Cos_k(x,y)+d_k Sin_k(x,y)\Big),\;\;  \text{$(x,y) \in \overline{B}_R\setminus\{\mathbf{0}\}$},
\end{equation}
where $Cos_k$ and $Sin_k$ are defined as in \eqref{expsincos0} and \eqref{expsincos}, respectively, and $c_k$ and $d_k$ are defined as in \eqref{ckdkdefn}. Then, $\Phi_n$ converges to $F$ uniformly on $B_{R}(\bold{0}) \setminus \{\bold{0}\}$ as $n \rightarrow \infty$. In particular,
\begin{align} 
| \Phi_n (x,y)-F(x,y)| \leq (2\pi)^{\alpha} \gamma_0 [f]_{C^{\alpha}([-\pi, \pi])}\cdot  \left( \frac{1}{n} \right)^{\alpha}\ln n,   \label{imptconvra}
\end{align}
where $\gamma_0$ is a constant as in Lemma \ref{jackson}.
\end{itemize}
\end{theorem}
\begin{proof}
(i) We have $\Delta \widetilde{u}_n = \Delta u_n =0$ in $\mathbb{R}^2 \setminus S$ by \eqref{harmonicun} and \eqref{unuequalit}. 
Since $\widetilde{u}_n$ is smooth in $\mathbb{R}^2$, we get $\Delta \widetilde{u}_n=0$ in $\mathbb{R}^2$.
Note that by \eqref{harmoe} and \eqref{unuequalit},
for each $\theta \in (-\pi, \pi)$
$$
\widetilde{u}_n(R \cos \theta, R \sin \theta)=u_n(R \cos \theta, R \sin \theta) =u_n(\mathbf{x}(R,\theta),\mathbf{y}(R,\theta)) = v_n (R, \theta),
$$
where $v_n$ is defined as in \eqref{fourierdefn}. Therefore, by the continuity of $\widetilde{u}_n$ and $v_n$
$$
\widetilde{u}_n(R \cos \theta, R \sin \theta)=v_n(R, \theta) \quad \text{ for all $\theta \in [-\pi, \pi]$}.
$$
By Proposition \ref{uniconth} and \eqref{ffunc}, $\widetilde{u}_n(R \cos(\cdot), R \sin(\cdot))$ converges to $g\big(R \cos(\cdot). R\sin (\cdot)\big)$
uniformly on $[-\pi, \pi]$ as $n \rightarrow \infty$, as desired. \\
(ii) Observe that for each $(x,y) \in \overline{B}_R \setminus S$, we get
$$
\Phi_n(x,y)=\frac{c_0}{2}+\sum_{k=1}^{n} R^k\Big(c_k \cos k\Theta(x,y)+d_k \sin k \Theta(x,y) \Big) =v_n \left(R, \Theta(x,y) \right),
$$
where $v_n$ is defined as in \eqref{fourierdefn}.
By Proposition \ref{uniconth}, we have
$$
|v_n(R, \theta) - f(\theta)| \leq (2\pi)^{\alpha} \gamma_0 [f]_{C^{\alpha}([-\pi, \pi])}  \left( \frac1n \right)^{\alpha} \ln n , \quad \text{ for all $\theta \in [-\pi, \pi]$}.
$$
Since the range of $\Theta$ on $\overline{B}_R \setminus S$ is $(-\pi, \pi)$,
$$
|\Phi_n(x,y)-f(\Theta(x,y))| \leq  (2\pi)^{\alpha} \gamma_0 [f]_{C^{\alpha}([-\pi, \pi])}  \cdot \left( \frac1n \right)^{\alpha} \ln n, \;\; \text{ for all $(x,y) \in \overline{B}_R \setminus S$}.
$$
Since $\Phi_n$ is continuous on $\overline{B}_R \setminus \{\bold{0}\}$ and $f \circ \Theta$ has a continuous extension $F$ on $\overline{B}_R \setminus \{\bold{0}\}$, \eqref{imptconvra} holds by Proposition \ref{holestim}, and hence the assertion follows.
\end{proof}

\centerline{}
\subsection{Uniform convergence of harmonic polynomials on $\overline{B}_R$}  \label{uniforconpol}
In this section, to derive our main result we use the idea of Niels Henrik Abel who developed the result named Abel's test by using the formula of summation by parts.
Although the formula below can be found in most elementary analysis books, the statement and its proof are left here for the reader's accessibility.

\begin{proposition}[Summation by parts] \label{sumabyparts}
Let $(a_n)_{n \geq 1}$ and $(b_n)_{n \geq 1}$ be sequences of real numbers, and let $B_n = \sum_{k=1}^n b_k$ for all $n \in \mathbb{N}$. Then for $n,m \in \mathbb{N}$ with $m \geq n$,
$$
\sum_{k=n}^m a_kb_k=a_{m+1}B_m-a_nB_{n-1}-\sum_{k=n}^m(a_{k+1}-a_k)B_k.
$$
\end{proposition}
\begin{proof}
We substitute $b_k=B_k-B_{k-1}$ into $\sum_{k=n}^ma_kb_k$. Note that
$$
\sum_{k=n}^ma_kb_k = \sum_{k=n}^m a_k(B_k-B_{k-1})=\sum_{k=n}^m a_kB_k -\sum_{k=n}^m a_kB_{k-1}.
$$
Rewriting the second term of right-hand side as
$$
\sum_{k=n}^m a_kB_{k-1} = \sum_{k=n-1}^{m-1} a_{k+1}B_{k} = \left(\sum_{k=n}^{m} a_{k+1}B_{k} \right) +a_nB_{n-1}-a_{m+1}B_m,
$$
we obtain that
$$
\sum_{k=n}^m a_kb_k=a_{m+1}B_m-a_nB_{n-1}-\sum_{k=n}^m(a_{k+1}-a_k)B_k.
$$
\end{proof}

\begin{theorem} \label{limitharmo}
Let $g \in C^\alpha(\partial B_R)$ for some $\alpha \in (0, 1]$ and $f$ be a function on $\mathbb{R}$ defined in \eqref{ffunc}  (Then, $f \in C^{\alpha}([-\pi, \pi])$ with $[f]_{C^{\alpha}[-\pi, \pi]} \leq R^{\alpha} [g]_{C^{\alpha}(\partial B_R)}$ by Proposition \ref{holestim}).
Let $\widetilde{u}_n$ be a polynomial defined in \eqref{tildeun}. Then, there exists $\widetilde{u} \in C(\overline{B}_R)$ such that
$\widetilde{u}_n$ converges to $\widetilde{u}$ uniformly on $\overline{B}_R$ as $n \rightarrow \infty$. Moreover, for each $n \geq e^{\frac{1}{\alpha}}$
\begin{align} 
&|\widetilde{u}(x,y)-\widetilde{u}_n(x,y)| \leq  2\gamma_0 (2\pi)^{\alpha}[f]_{C^\alpha([-\pi, \pi])} \left(\frac{\sqrt{x^2+y^2}}{R}\right)^{n+1} \left( \frac{1}{n} \right)^{\alpha} \ln n, \quad \forall \text{$(x,y) \in \overline{B}_R$},   \label{mainestimer}
\end{align}
where $\gamma_0$ is the constant in Lemma \ref{jackson}.
In particular, $\widetilde{u}(x,y) = g(x,y)$ for all $(x,y) \in \partial B_R$. 
\end{theorem}
\begin{proof}
Define $S_n: \overline{B}_R \setminus \{ \bold{0}\} \rightarrow \mathbb{R}$ given by 
\begin{equation*}
\quad S_n(x,y) := \sum_{k=1}^n \left(\sqrt{x^2+y^2}\right)^k \left(c_k Cos_k(x,y)+d_k Sin_k(x,y)\right), \quad (x,y) \in \overline{B}_R(\mathbf{0})\setminus\{\mathbf{0}\},
\end{equation*}
where $Cos_k$ and $Sin_k$ are defined as in \eqref{expsincos0}  and \eqref{expsincos}, respectively and $c_k$ and $d_k$ are defined as in \eqref{ckdkdefn}. For each $n \geq 1$, define the function $\tau_n$ on $\overline{B}_R$ given by
$$
\tau_n(x,y) := \left(\frac{\sqrt{x^2+y^2}}{R}\right)^n, \quad (x,y) \in \overline{B}_R.
$$
Then, for each $(x,y) \in \overline{B}_R(\mathbf{0})\setminus\{\mathbf{0}\}$
$$
S_n(x,y)= \sum_{k=1}^n \tau_k(x,y) \cdot R^k\left(c_k Cos_k(x,y)+d_k Sin_k(x,y)\right). 
$$
Let $\widehat{\Phi}_n := \Phi_n - c_0/2$ and $\widehat{F} := F - c_0/2$, where $\Phi_n$ is the function defined in \eqref{Phidefn} on $\overline{B}_R \setminus \{\mathbf{0}\}$ and $F$ is the function given in Theorem \ref{polyfinal}(ii) on $B_R \setminus \{\mathbf{0}\}$.
Now let $m \geq n+1$. Using the summation by parts in Proposition \ref{sumabyparts}, we obtain that  for each $(x,y) \in \overline{B}_R\setminus\{\mathbf{0}\}$,
\begin{align} 
&S_m(x,y)-S_n(x,y) =\sum_{k=n+1}^m \tau_k(x,y) \cdot R^k\left(c_k Cos_k(x,y)+d_k Sin_k(x,y)\right) \nonumber \\
& = \ \tau_{m+1}(x,y) \widehat{\Phi}_m(x,y)-\tau_{n+1}(x,y) \widehat{\Phi}_n(x,y)-\sum_{k=n+1}^m(\tau_{k+1}(x,y)-\tau_k(x,y)) \widehat{\Phi}_k(x,y) \nonumber \\
&  =  \ \tau_{m+1}(x,y)\widehat{\Phi}_m(x,y)-\tau_{n+1}(x,y)\widehat{\Phi}_n(x,y)-\sum_{k=n+1}^m(\tau_{k+1}(x,y)-\tau_k(x,y))\widehat{\Phi}_k(x,y) \nonumber \\
&\qquad \; -\tau_{m+1}(x,y)\widehat{F}(x,y)+\tau_{n+1}(x,y)\widehat{F}(x,y)+\sum_{k=n+1}^m(\tau_{k+1}(x,y)-\tau_k(x,y))\widehat{F}(x,y) \nonumber \\
& =-\tau_{m+1}(x,y)(\widehat{F}(x,y)-\widehat{\Phi}_m(x,y))+\tau_{n+1}(x,y)(\widehat{F}(x,y)-\widehat{\Phi}_n(x,y)) \nonumber \\
&\qquad \qquad  + \sum_{k=n+1}^m(\tau_{k+1}(x,y)-\tau_k(x,y))(\widehat{F}(x,y)-\widehat{\Phi}_k(x,y)). \label{sumbyparts}
\end{align}
By Theorem \ref{polyfinal}(ii),  for all  $k \geq 1$ and  $(x,y) \in \overline{B}_R(\mathbf{0}) \setminus \{ \bold{0}\}$
$$
|\widehat{F}(x,y)-\widehat{\Phi}_k(x,y)| \leq \gamma_0 (2\pi)^{\alpha}[f]_{C^\alpha([-\pi, \pi])} \ln k \, \left( \frac{1}{k} \right)^{\alpha}.
$$
Meanwhile, observe that the function $x \in (0, \infty) \mapsto \left( \ln x \right) \left( \frac{1}{x} \right)^{\alpha}$ is decreasing on $[e^{\frac{1}{\alpha}}, \infty )$, and hence
$$
\max_{n+1 \leq k} \left(\ln k \right) \left( \frac{1}{k} \right)^{\alpha} \leq  \left(\ln n \right) \left( \frac{1}{n} \right)^{\alpha}, \quad \text{ for all $n \geq e^{\frac{1}{\alpha}}$}.
$$
Thus,  it follows from \eqref{sumbyparts} that for each $(x,y) \in \overline{B}_R \setminus \{ \bold{0} \}$ and $n \geq e^{\frac{1}{\alpha}}$
\begin{align}
&|\widetilde{u}_m(x,y)-\widetilde{u}_n(x,y)|= \ |S_m(x,y)-S_n(x,y)| \nonumber \\
&\qquad \qquad  \leq \ \tau_{m+1}(x,y)|F(x,y)-\Phi_m(x,y)|+\tau_{n+1}(x,y)|F(x,y)-\Phi_n(x,y)| \nonumber \\
&\qquad \qquad \qquad  +\sum_{k=n+1}^m(\tau_{k}(x,y)-\tau_{k+1}(x,y))|F(x,y)-\Phi_k(x,y)|  \nonumber \\
& \qquad \qquad \leq \gamma_0 (2\pi)^{\alpha} [f]_{C^\alpha([-\pi, \pi])}\bigg( \tau_{m+1}(x,y) (\ln m) \left(\frac{1}{m} \right)^{\alpha}  \nonumber \\
& \qquad \qquad \qquad +\tau_{n+1}(x,y) (\ln n) \left(\frac{1}{n} \right)^{\alpha} +\Big(\tau_{n+1}(x,y) - \tau_{m+1}(x,y) \Big) \max_{n+1 \leq k \leq m} (\ln k)  \left(\frac{1}{k} \right)^{\alpha}\bigg)  \nonumber  \\
& \leq \gamma_0 (2\pi)^{\alpha} [f]_{C^\alpha([-\pi, \pi])}\bigg( \tau_{m+1}(x,y) (\ln m) \left(\frac{1}{m} \right)^{\alpha}   +\tau_{n+1}(x,y) (\ln n) \left(\frac{1}{n} \right)^{\alpha} +(\tau_{n+1}(x,y)\,  (\ln n)  \left(\frac{1}{n} \right)^{\alpha}\bigg). \label{estimpartial}
\end{align}
\eqref{estimpartial} also holds for every $(x,y) \in \overline{B}_R$ since we see that $\widetilde{u}_m(0,0)-\widetilde{u}_n(0,0)=0-0=0$ for any $m,n \in \mathbb{N}$. Thus, $(\widetilde{u}_n)_{n \geq 1}$ is a Cauchy sequence in $C(\overline{B}_R)$ and hence there exists $\widetilde{u} \in C(\overline{B}_R)$ such that $\lim_{n \rightarrow \infty} \widetilde{u}_n = \widetilde{u}$ in $C(\overline{B}_R)$. Therefore, letting $m \rightarrow \infty$ in \eqref{estimpartial}, we get \eqref{mainestimer}.
The last assertion follows from Theorem \ref{polyfinal}(i).
\end{proof}
\centerline{}

\subsection{Constructing a solution in $C^{\infty} (B_R)\cap C^\alpha(\overline{B}_R)$} \label{refsec}
In this section, we will finally show that $\widetilde{u} \in C(\overline{B}_R)$ in Theorem \ref{limitharmo} satisfies  $\widetilde{u} \in C^{\infty}(B_R)$ and solves \eqref{maineq}. Indeed, one can immediately get the desired result using the result, Weyl's lemma  (\cite[Corollary 2.2.1]{jo13}). Precisely, we observe by Theorem \ref{polyfinal}(i) that for each $n \geq 1$ $\Delta \widetilde{u}_n =0$ in $B_R$, where $\widetilde{u}_n$ is defined in \eqref{tildeun}. Then, using integration by parts 
$$
0=\int_{B_R} \Delta \widetilde{u}_n \cdot \varphi\, dx = \int_{B_R} \widetilde{u}_n \cdot \Delta \varphi\, dx, \quad \text{ for all $\varphi \in C_0^{\infty}(B_R)$},
$$
Now Theorem \ref{limitharmo} yields that
$$
\int_{B_R} \widetilde{u} \cdot \Delta \varphi dx = \lim_{n \rightarrow \infty} \int_{B_R} \widetilde{u}_n \cdot \Delta \varphi dx = 0 \quad \text{ for all $\varphi \in C_0^{\infty}(B_R)$}.
$$
Thus, using Weyl's lemma, we ultimately obtain that $\widetilde{u} \in C^{\infty}(B_R)$ and $\Delta \widetilde{u}=0$ on $B_R$. (We specifically refer to \cite{BKR01, BKR09, BKRS15} for the results of generalizing Weyl's lemma to more general elliptic and parabolic operators.)
As mentioned in the introduction, we will achieve the same result above without using the multi-variable Riemann (or Lebesgue) integral or any advanced results requiring significant integral calculus, such as Weyl's lemma and Gauss's theorem. Instead, we will rely only on very elementary results for classical derivatives and uniform convergence to show $\widetilde{u} \in C^{\infty}(B_R)$ and $\Delta \widetilde{u}=0$ on $B_R$. Moreover, using the harmonic polynomial approximation, we will show the non-trivial uniform estimates for the derivative of our solution with arbitrary order in terms of the $L^1$-boundary data.
\centerline{}
\begin{theorem} \label{construsol}
Let $g \in C^{\alpha}(\partial B_R)$ for some $\alpha \in (0,1]$ and $f$ be a function on $\mathbb{R}$ defined in \eqref{ffunc}  (Then, $f \in C^{\alpha}([-\pi, \pi])$ with $[f]_{C^{\alpha}[-\pi, \pi]} \leq R^{\alpha} [g]_{C^{\alpha}(\partial B_R)}$ by Proposition \ref{holestim}).
Let $\widetilde{u} \in C(\overline{B}_R)$ in Theorem \ref{limitharmo} and $(\widetilde{u}_n)_{n \geq 1}$ be a sequence of functions defined in \eqref{tildeun}. 
Then, for each $s \in \mathbb{N}$ and $r \in (0, R)$, it holds that 
$$
\lim_{n \rightarrow \infty} \widetilde{u}_n = \widetilde{u} \;\; \text{ in }\; C^{s}(\overline{B}_r).
$$
Moreover, $\widetilde{u} \in C(\overline{B}_R) \cap C^{\infty}(B_R)$ and $\widetilde{u}$ is a solution to \eqref{maineq}. In particular, for each $\alpha_1, \alpha_2 \in \mathbb{N} \cup \{0\}$ and $r \in (0, R)$
\begin{align} 
|D^{(\alpha_1, \alpha_2)} \widetilde{u} (x,y)| &\leq 
 D^{(\alpha_1, \alpha_2)} \left(\frac{-1}{2\pi}\int_{-\pi}^{\pi} f(\theta)\,d\theta \right) + \frac{R(\alpha_1+\alpha_2)!}{\pi (R-r)^{\alpha_1+\alpha_2+1}}  \int_{-\pi}^{\pi}|f(\theta)| d \theta, \nonumber \\
 &= 
 D^{(\alpha_1, \alpha_2)} \left(\frac{-1}{2\pi R}\int_{\partial B_R} |g|\,ds \right) + \frac{(\alpha_1+\alpha_2)!}{\pi (R-r)^{\alpha_1+\alpha_2+1}}  \int_{\partial B_R}|g| d s, \;\; \;\forall (x,y) \in \overline{B}_r.  \label{derivativestim}
\end{align}
\end{theorem}
\begin{proof}
 For each $k \in \mathbb{N}$, we define polynomials $P_k(x,y)$, $Q_k(x,y)$, and $h_k(x,y)$ given by
\begin{align*}
P_k(x,y) &:= \sum_{j \; \text{even}  \atop 0 \leq j \leq k}^{}(-1)^\frac{j}{2}\binom{k}{j}x^{k-j}y^{j}, \nonumber \\
Q_k(x,y) &:= \sum_{j \; \text{odd}  \atop 0 \leq j \leq k}^{}(-1)^\frac{j-1}{2}\binom{k}{j}x^{k-j}y^{j},  \nonumber \\
h_k(x,y) &:= c_k P_k(x,y) + d_k Q_k(x,y), \quad (x,y) \in \mathbb{R}^2, 
\end{align*}
where $c_k$ and $d_k$ are defined as in \eqref{ckdkdefn}. Moreover, we define $P_0:=1$ and $Q_0:=0$.
Then by the definition of $\widetilde{u}_n$ in \eqref{tildeun}, we get
$$
\widetilde{u}_n(x,y)-\frac{c_0}{2}=\sum_{k=1}^n h_k(x,y), \quad \text{ for all }\; (x,y) \in \mathbb{R}^2.
$$
Let us compute derivatives of $P_k$ and $Q_k$. For each $(x,y) \in \mathbb{R}^2$ we obtain that
\begin{align*}
\partial_x P_k(x,y)&=\partial_x\sum_{j \; \text{even}  \atop 0 \leq j \leq k}^{}(-1)^\frac{j}{2}\binom{k}{j}x^{k-j}y^{j}=\sum_{j \; \text{even}  \atop 0 \leq j \leq k-1}^{}(-1)^\frac{j}{2}\binom{k}{j}(k-j)x^{k-j-1}y^{j} \\
& =k\sum_{j \; \text{even}  \atop 0 \leq j \leq k-1}^{}(-1)^\frac{j}{2}\binom{k-1}{j}x^{k-j-1}y^{j}=kP_{k-1}(x,y), \\[7pt]
\partial_y P_k(x,y)&=\partial_y\sum_{j \; \text{even}  \atop 0 \leq j \leq k}^{}(-1)^\frac{j}{2}\binom{k}{j}x^{k-j}y^{j}=\sum_{j \; \text{even}  \atop 1 \leq j \leq k}^{}(-1)^\frac{j}{2}\binom{k}{j}jx^{k-j}y^{j-1} \\
&=k\sum_{j \; \text{even}  \atop 1 \leq j \leq k}^{}(-1)^\frac{j}{2}\binom{k-1}{j-1}x^{k-j}y^{j-1}=-k\sum_{j \; \text{odd}  \atop 0 \leq j \leq k-1}^{}(-1)^\frac{j-1}{2}\binom{k-1}{j}x^{k-j-1}y^{j}=-kQ_{k-1}(x,y),
\end{align*}
\begin{align*}
\partial_x Q_k(x,y)&=\partial_x\sum_{j \; \text{odd}  \atop 0 \leq j \leq k}^{}(-1)^\frac{j-1}{2}\binom{k}{j}x^{k-j}y^{j}=\sum_{j \; \text{odd}  \atop 0 \leq j \leq k-1}^{}(-1)^\frac{j-1}{2}\binom{k}{j}(k-j)x^{k-j-1}y^{j} \\
&=k\sum_{j \; \text{odd}  \atop 0 \leq j \leq k-1}^{}(-1)^\frac{j-1}{2}\binom{k-1}{j}x^{k-j-1}y^{j}=kQ_{k-1}(x,y),\\
\partial_y Q_k(x,y)&=\partial_y\sum_{j \; \text{odd}  \atop 0 \leq j \leq k}^{}(-1)^\frac{j-1}{2}\binom{k}{j}x^{k-j}y^{j}=\sum_{j \; \text{odd}  \atop 1 \leq j \leq k}^{}(-1)^\frac{j-1}{2}\binom{k}{j}jx^{k-j}y^{j-1} \\
&=k\sum_{j \; \text{odd}  \atop 1 \leq j \leq k}^{}(-1)^\frac{j-1}{2}\binom{k-1}{j-1}x^{k-j}y^{j-1}=k\sum_{j \; \text{even}  \atop 0 \leq j \leq k-1}^{}(-1)^\frac{j}{2}\binom{k-1}{j}x^{k-j-1}y^{j}=kP_{k-1}(x,y).
\end{align*}
Thus, for each $\alpha_1,\alpha_2 \in \mathbb{N} \cup \{0\}$,
\begin{align}
&D^{(\alpha_1,\alpha_2)}h_k(x,y) \nonumber \\
&=\begin{cases}
0 & \text{ if } k<\alpha_1+\alpha_2, \\
(-1)^{\frac{\alpha_2}{2}}\frac{k!}{(k-\alpha_1-\alpha_2)!}(c_k P_{k-\alpha_1-\alpha_2}(x,y)+d_k Q_{k-\alpha_1-\alpha_2}(x,y)) & \text{ if } \alpha_2 \equiv 0 \text{ (mod 2),} \\
(-1)^{\frac{\alpha_2-1}{2}} \frac{k!}{(k-\alpha_1-\alpha_2)!}(-c_kQ_{k-\alpha_1-\alpha_2}(x,y)+d_kP_{k-\alpha_1-\alpha_2}(x,y)) & \text{ if } \alpha_2  \equiv 1 \text{ (mod 2),}
\end{cases} \label{casedefn}
\end{align}
Let $m=k-\alpha_1-\alpha_2$ when $k \geq \alpha_1+\alpha_2$. Using the expressions of $c_k$ and $d_k$ expressed as in \eqref{ckdkdefn} and the expression in \eqref{expsincos}, we get
\begin{align}
|c_kP_{m}(x,y)+d_kQ_{m}(x,y)| &\leq |c_k||P_m(x,y)|+|d_k||Q_m(x,y)| \nonumber \\
&\leq \left|\frac{1}{R^k\pi}\int_{-\pi}^{\pi}f(\theta)\cos\theta d\theta\right| \left|(\sqrt{x^2+y^2})^{m} Cos_m (x,y)\right| \nonumber \\
&\qquad \;\;+\left|\frac{1}{R^k\pi}\int_{-\pi}^{\pi}f(\theta)\sin\theta d\theta\right| \left| (\sqrt{x^2+y^2})^m Sin_m(x,y)\right| \nonumber  \\
&\leq \frac{(\sqrt{x^2+y^2})^m}{R^k\pi}\int_{-\pi}^{\pi}|f(\theta)|d\theta, \quad \text{ for all }(x,y) \in \overline{B}_R(\mathbf{0}) \setminus \{\mathbf{0}\}, \label{ineqfintes}
\end{align}
where we used the identity, $\cos \theta \cos \beta +\sin\theta \sin \beta = \cos(\theta-\beta)$ for any $\beta \in \mathbb{R}$.
The inequality \eqref{ineqfintes} holds for all $(x,y) \in \overline{B}_R(\mathbf{0})$ since $P_m(\bold{0})=Q_m(\bold{0})=0$.
Likewise, we have
$$
|-c_kQ_{m}(x,y)+d_kP_{m}(x,y)| \leq  \frac{(\sqrt{x^2+y^2})^m}{R^k\pi}\int_{-\pi}^{\pi}|f(\theta)|d\theta, \quad \text{ for all }(x,y) \in \overline{B}_R(\mathbf{0}) \setminus \{\mathbf{0}\}. 
$$
Now fix $r \in (0,R)$. Then, for each $\alpha_1,\alpha_2 \in \mathbb{N} \cup \{0\}$ and $k \in \mathbb{N}$, define
\begin{align} \label{bound}
M_k^{(\alpha_1,\alpha_2)}:=\begin{cases}
0 & \text{ if } k<\alpha_1+\alpha_2, \\
\displaystyle \frac{k!}{(k-\alpha_1-\alpha_2)! \pi R^{\alpha_1+\alpha_2}} \left(\frac{r}{R} \right)^{k-\alpha_1-\alpha_2} \int_{-\pi}^{\pi} |f(\theta)| d\theta & \text{ otherwise}. 
\end{cases} 
\end{align}
Then, \eqref{casedefn} and \eqref{bound} implies that
\begin{equation*} 
|D^{(\alpha_1,\alpha_2)}h_k(x,y)| \leq M_k^{(\alpha_1,\alpha_2)}, \quad \text{ for all $(x,y) \in B_R$}.
\end{equation*}
Now let $\alpha_1,\alpha_2 \in \mathbb{N} \cup \{0\}$. 
Then,
\begin{align*}
\sum_{k=1}^{\infty} M_k^{(\alpha_1, \alpha_2)} = D^{(\alpha_1, \alpha_2)}\left( -\frac{1}{\pi} \int_{-\pi}^{\pi} |f(\theta)| d \theta \right)  +
\frac{(\alpha_1+\alpha_2)!}{\left( 1-\frac{r}{R} \right)^{\alpha_1+\alpha_2+1}}  \left(\frac{1}{\pi R^{\alpha_1+\alpha_2}} \int_{-\pi}^{\pi} |f(\theta)| d\theta \right).
\end{align*}
Thus, it follows from the Weierstrass $M$-test that $D^{(\alpha_1,\alpha_2)}\widetilde{u}_n$ uniformly converges to a continuous function in $\overline{B}_r$. Thus, we finially obtain that for each $s \in \mathbb{N}$, $\widetilde{u}_n$ converges to a function in $C^s(\overline{B}_r)$. 
Since  $r \in (0,R)$ is arbitrarily chosen, $\widetilde{u} \in C^{\infty}(B_R)$ by Theorem \ref{limitharmo}.
Since $\Delta \widetilde{u}_n=0$ on $B_R$ by Theorem \ref{polyfinal}(i), we get $\Delta \widetilde{u}=0$ on $B_R$. Moreover, $\widetilde{u}=g$ on $\partial B_R$ by Theorem \ref{limitharmo}. Finally, for each $(x,y) \in \overline{B}_r$ we obtain that
\begin{align*}
|D^{(\alpha_1, \alpha_2)} \widetilde{u}(x,y)| &= \lim_{n \rightarrow \infty} |D^{(\alpha_1,\alpha_2)}\widetilde{u}_n(x,y)| \leq  \left|D^{(\alpha_1, \alpha_2)  }\frac{c_0}{2}\right|+\lim_{n \rightarrow \infty} \sum_{k=1}^{n} |D^{(\alpha_1, \alpha_2)} h_k(x,y) | \\
& \leq   D^{(\alpha_1, \alpha_2)}\left( -\frac{1}{2\pi} \int_{-\pi}^{\pi} |f(\theta)| d \theta \right) + \frac{R(\alpha_1+\alpha_2)!}{\pi (R-r)^{\alpha_1+\alpha_2+1}}  \int_{-\pi}^{\pi}|f(\theta)| d \theta,
\end{align*}
as desired.
\end{proof}
\centerline{}
\noindent
Independent of the existence result for \eqref{maineq}, the uniqueness result for \eqref{maineq} can be proven in various ways. In particular, the maximum principle immediately leads to the uniqueness of the solution to \eqref{maineq}. The strong maximum principle is derived through the mean value property, which requires a basic understanding of integral calculus. On the other hand, the weak maximum principle can be derived directly using the second-order derivative test for single-variable functions, which immediately implies the uniqueness of solutions. For the sake of accessibility for readers, we present here the statement of the weak maximum principle.

\begin{theorem}[{\cite[Corollary 3.27]{AU23}}, the weak maximum principle] \label{weakmaxi}
Assume that $w \in C^2(B_R)\cap C(\overline{B}_R)$ is a solution to \eqref{maineq}. Then,
$$
\min_{y \in \partial B_R} g(y) \leq w(x) \leq \max_{y \in \partial B_R}g (y), \quad \text{ for all $x \in \overline{B}_R$}.
$$
In particular, the uniqueness of the solutions to \eqref{maineq} holds. 
\end{theorem}
\centerline{}
\noindent
Now, we are ready to show Theorem \ref{mainthmintro}(i)
\begin{theorem} \label{mainthmi}
Let  $g \in C(\partial B_R)$. Then, there exists a unique solution $\widetilde{u} \in C(\overline{B}_R) \cap C^{\infty}(B_R)$ to \eqref{maineq}. Moreover, for each $r \in [0, R)$, $(x,y) \in \overline{B}_r$ and $\alpha_1, \alpha_2 \in \mathbb{N} \cup \{0\}$, \eqref{derivativestim} holds.
\end{theorem}
\noindent
\begin{proof}
Let $(g_n)_{n \geq 1}$ be a sequence of smooth functions on $\mathbb{R}^2$ such that $g_n$ converges to $g$ uniformly on $\partial B_R$. Let $\mathcal{T}g_n \in C(\overline{B}_R) \cap C^{\infty}(B_R)$ be a unique solution to \eqref{maineq} where $g$ is replaced by $g_n$,  constructed as in Theorem \ref{construsol}. Then, by the weak maximum principle (see Theorem \ref{weakmaxi}), 
$$
\| \mathcal{T} g_n \|_{C(\overline{B}_R)} \leq \|g_n \|_{C(\partial B_R)}.
$$
Using the completeness argument, there exists $\widetilde{u} \in C(\overline{B}_R)$ such that $\mathcal{T}g_n$ converges to $\widetilde{u}$ uniformly on $\overline{B}_R$, and hence $\widetilde{u}(x) =g(x)$ for all $x \in \partial B_R$. Now let $r \in (0, R)$ and $\alpha_1, \alpha_2 \in \mathbb{N} \cup \{0\}$ with $\alpha_1+\alpha_2 \geq 1$. Then, Theorem \ref{construsol} yields that
\begin{align*} \label{uniformat2}
\|D^{(\alpha_1, \alpha_2)} \mathcal{T}g_n- D^{(\alpha_1, \alpha_2)} \mathcal{T}g_m\|_{C(\overline{B}_r)} &\leq \frac{(\alpha_1+\alpha_2)!}{\pi (R-r)^{\alpha_1+\alpha_2+1}}  \int_{\partial B_R} |g_n-g_m| d s. 
\end{align*}
Using the completeness argument for $C^{\alpha_1+\alpha_2}(\overline{B}_r)$, we get $\widetilde{u} \in C^{\alpha_1+\alpha_2}(\overline{B}_r)$ such that
$\mathcal{T} g_n$ converges to $\widetilde{u}$ in $C^{\alpha_1+\alpha_2}(\overline{B}_r)$. Therefore, $\widetilde{u} \in C^{\infty}(B_R)$.
Moreover, $\Delta \widetilde{u} = 0$ on $B_R$ since $\Delta \mathcal{T}g_n = 0$ on $B_R$. Thus, $\widetilde{u}$ is a unique solution to \eqref{maineq} by Theorem \ref{weakmaxi}. Finally, since $\lim_{n \rightarrow \infty}g_n = g$ in $C(\partial B_R)$ and $\lim_{n \rightarrow \infty} D^{(\alpha_1, \alpha_2)} \mathcal{T}g_n = \widetilde{u}$ in $C(\overline{B}_r)$ and \eqref{derivativestim} holds where $g$ and $\widetilde{u}$ are replaced by $g_n$ and $\mathcal{T}g_n$, the last assertion follows. 
\end{proof}

\begin{corollary} \label{improconst}
\begin{itemize}
\item[(i)]
Let $u \in C(\overline{B}_L(\bold{x}_0)) \cap C^{\infty}(B_L(\bold{x}_0))$ satisfy $\Delta u= 0$ in $B_L(\bold{x}_0)$, where $\bold{x}_0 \in \mathbb{R}^2$ and $L>0$. Then, for any $r \in [0,L)$ 
\begin{equation} \label{uniestimtrans}
\|D^{(\alpha_1, \alpha_2)} u\|_{C(\overline{B}_r(\bold{x}_0))}  \leq   D^{(\alpha_1, \alpha_2)} \left(\frac{-1}{2\pi L}\int_{\partial B_L(\bold{x}_0)} u\,ds \right) + \frac{(\alpha_1+\alpha_2)!}{\pi (L-r)^{\alpha_1+\alpha_2+1}}  \int_{\partial B_L(\bold{x}_0)}|u| ds,
\end{equation}
where $\|D^{(\alpha_1, \alpha_2)} u\|_{C(\overline{B}_0(\bold{x}_0))}:=|D^{(\alpha_1, \alpha_2)} u(\bold{x}_0)|$.
In particular, if $\alpha_1+\alpha_2 \geq 1$, then
\begin{align} \label{intristim}
|D^{(\alpha_1, \alpha_2)} \widetilde{u}(\bold{x}_0)| &\leq \frac{(\alpha_1+\alpha_2)!}{\pi L^{\alpha_1+\alpha_2+1}}  \int_{\partial B_L(\bold{x}_0)}|u| d s \leq  \frac{2(\alpha_1+\alpha_2)!}{L^{\alpha_1+\alpha_2}} \| u\|_{C(\overline{B}_L(\bold{x}_0))}. 
\end{align}
\item[(ii)]
Let  $g \in C(\partial B_R)$ and $u$ be a unique solution to \eqref{maineq} as in Theorem \ref{mainthmi}. Then for any $\bold{x} \in B_R$ and $\alpha_1+\alpha_2 \geq 1$, we have
\begin{align}
|D^{(\alpha_1, \alpha_2)} \widetilde{u}(\bold{x})| &\leq \frac{(\alpha_1+\alpha_2)!}{\pi (R-\| \bold{x}\|)^{\alpha_1+\alpha_2+1}}  \int_{\partial B_{R-\|\bold{x}\|}(\bold{x})}|u| d s \leq  \frac{2(\alpha_1+\alpha_2)!}{(R-\|\bold{x}\|)^{\alpha_1+\alpha_2}} \| u\|_{C(\overline{B}_{R-\|\bold{x}\|}(\bold{x}))} \nonumber \\ 
&\leq  \frac{2(\alpha_1+\alpha_2)!}{(R-\|\bold{x}\|)^{\alpha_1+\alpha_2}} \| g\|_{C(\partial \overline{B}_R) }. \label{intristim2nd}
\end{align}
\end{itemize}
\end{corollary}
\begin{proof}
(i) \eqref{uniestimtrans}  and \eqref{intristim} directly follow from \eqref{derivativestim} in Theorem \ref{construsol} with the translation from $\bold{x}_0$ to $\bold{0}$ and letting $r \rightarrow 0+$. \\
(ii) Note that $u \in C(\overline{B}_L(\bold{x}_0)) \cap C^{\infty}(B_L(\bold{x}_0))$ where $L:=R-\|\bold{x}\|$. Thus, the assertion follows from \eqref{intristim} and the weak maximum principle (see Theorem \ref{weakmaxi}).
\end{proof}

\begin{remark} \label{l1estimconsim}
Our results are quantitatively improved than \cite{HL11} and \cite{Kr96} in terms of constant values in \eqref{intristim}, \eqref{intristim2nd}, respectively.
For instance, if $\alpha_1+\alpha_2 \geq 2$, then the constants $\frac{2(\alpha_1+\alpha_2)!}{L^{\alpha_1+\alpha_2}}$
  in the right-hand side of \eqref{intristim} are smaller than the constants 
$\frac{2^{\alpha_1+\alpha_2} e^{\alpha_1+\alpha_2-1} (\alpha_1+\alpha_2)!}{L^{\alpha_1+\alpha_2}}$
 in \cite[Proposition 1.13]{HL11}. Also, if $\alpha_1+\alpha_2 \geq 2$, then the constants $\frac{2(\alpha_1+\alpha_2)!}{(R-\|\bold{x}\|)^{\alpha_1+\alpha_2}}$ in the right-hand side of \eqref{intristim2nd} is smaller than the constants $\frac{2^{\alpha_1+\alpha_2}(\alpha_1+\alpha_2)^{\alpha_1+\alpha_2}}{(R-\|\bold{x}\|)^{\alpha_1+\alpha_2}}$ in \cite[Theorem 2.5.2]{Kr96}.
\end{remark}

\centerline{}

\begin{theorem}  \label{holomoresul}
Assume $g \in C(\partial B_R)$ and let $\widetilde{u} \in C^{\infty}(B_R) \cap C(\overline{B}_R)$  be a (unique) solution in Theorem \ref{mainthmi}. Then, $\widetilde{u}$ is analytic on $B_R$. Precisely, for each $\bold{x}_0=(x_0,y_0) \in B_R$, $\kappa \in (0,\frac{1}{3})$, $\bold{h} =(h_1, h_2) \in \mathbb{R}^2$ with $\| \bold{h}\|\in [0, \kappa L)$ and $L:= R-\|\bold{x}_0\|$, \eqref{taylorexpress} and \eqref{analytiestim} hold.

\end{theorem}
\begin{proof}
Let $\bold{x}_0 =(x_0, y_0) \in B_R$ and $\kappa \in (0,\frac{1}{3})$ be fixed and write $L:=R-\|\bold{x}_0\|>0$.  Let $\bold{h} =(h_1, h_2) \in \mathbb{R}^2$ with $\| \bold{h}\|\in [0, \kappa L)$ and set
$$
\varepsilon:= \frac{\kappa L-\|\bold{h}\|}{L} \in (0,\kappa].
$$
First, note that $B_{(1+\varepsilon) \| \bold{h}\|}(\bold{x}_0) \subset B_{\kappa L}( \bold{x}_0)$ since $\|(1+\varepsilon) \bold{h}\| < \|\bold{h}\| + \varepsilon L = \kappa L$.  Moreover, $\overline{B}_{\kappa L}(\bold{x}_0)  \subset B_{L}(\bold{x}_0) \subset B_R$. Define a function
$\varphi:(-1-\varepsilon,1+\varepsilon) \rightarrow \mathbb{R}$ given by
$$
\varphi(t):= \widetilde{u}(\bold{x}_0+t \bold{h}), \quad t \in (-1-\varepsilon,\,1+\varepsilon).
$$
Then, $\varphi \in C^{\infty}((-1-\varepsilon, 1+\varepsilon))$ since $\widetilde{u} \in C^{\infty}(B_R)$.
Now, we claim that
\begin{equation} \label{taylorone}
\varphi(t) = \sum_{k=0}^{\infty} \frac{\varphi^{(k)}(0)}{k!} t^k, \qquad \text{ for all $t \in (-1-\varepsilon,\,1+\varepsilon)$}.
\end{equation}
Indeed, for each $n \geq 1$, let 
\begin{equation} \label{taylorremain1}
R_n(t):=\varphi(t)-\sum_{k=0}^{n-1} \frac{\varphi^{(k)}(0)}{k!} t^k, \;\; \quad t \in (-1-\varepsilon, \,1+\varepsilon). 
\end{equation}
Now, fix $t \in (-1-\varepsilon,\, 1+\varepsilon)$.
Then, using Taylor's theorem, there exists $\zeta_t \in (0,1)$ such that
\begin{equation} \label{taylorremain2}
R_n(t)= \frac{\varphi^{(n)}(t \zeta_t)}{n!}.
\end{equation}
Meanwhile, the chain rule implies that for each $n \in \mathbb{N}$ and $s \in (-1-\varepsilon, 1+\varepsilon)$, we have
$$
\varphi^{(n)}(s) = (h_1 \partial_1 +h_2 \partial_2)^n \widetilde{u} (\bold{x}_0+s\bold{h}) = \sum_{\alpha_1+\alpha_2=n}  \binom{n}{\alpha_1} h_1^{\alpha_1} h_2^{\alpha_2} D^{(\alpha_1, \alpha_2)} \widetilde{u}(\bold{x}_0+s  \bold{h}).
$$
Observe that $\bold{x}_0+t \zeta_t \bold{h} \in B_{\kappa L}(\bold{x}_0)$ since $\| t \zeta_t \bold{h}\|<(1+\varepsilon) \|\bold{h}\|< \|\bold{h}\|+\varepsilon L = \kappa L$.
By \eqref{derivativestim} in Theorem \ref{construsol}, we get
\begin{align}
|R_n(t)|&=\frac{|\varphi^{(n)}(t\zeta_t) |}{n!} \leq \frac{2^n}{n!} \sqrt{h_1^2+h_2^2}^n \max_{\alpha_1+\alpha_2=n} \Big(D^{(\alpha_1, \alpha_2)} \widetilde{u}(\bold{x}_0+t \zeta_t \bold{h}) \Big)  \nonumber \\
& \leq \left(\frac{2 \| \bold{h}\|}{L-\kappa L}\right)^n \frac{1}{\pi( L-\kappa L)}  \int_{\partial B_{L} (\bold{x}_0)}|u| ds  \nonumber \\
&= \left(\frac{2 \kappa }{1-\kappa }\right)^n \frac{1}{\pi( L-\kappa L)}  \int_{\partial B_{L} (\bold{x}_0)}|\widetilde{u}| ds\longrightarrow 0, \quad \text{ as $n \rightarrow \infty$},	\label{lastestimer}
\end{align}
so that the claim is shown.
By substituting $t=1$ in \eqref{taylorone}, \eqref{taylorremain1}, \eqref{taylorremain2} and \eqref{lastestimer}, it follows that
$$
\widetilde{u}(\bold{x}_0+\bold{h})=\varphi(1)=\sum_{k=0}^{\infty} \frac{\varphi^{(k)}(0)}{k!} = \sum_{k=0}^{\infty}\left( \sum_{\alpha_1+\alpha_2=k}  \frac{D^{(\alpha_1, \alpha_2)} \widetilde{u}(\bold{x}_0) }{\alpha_1 ! \alpha_2!} h_1^{\alpha_1} h_2^{\alpha_2} \right),
$$
$$
\left| \widetilde{u}(\bold{x}_0 + \bold{h})-\sum_{k=0}^{n-1}\left( \sum_{\alpha_1+\alpha_2=k}  \frac{D^{(\alpha_1, \alpha_2)} \widetilde{u}(\bold{x}_0) }{\alpha_1 ! \alpha_2!} h_1^{\alpha_1} h_2^{\alpha_2} \right)  \right| \leq \left(\frac{2 \kappa }{1-\kappa }\right)^n \frac{1}{\pi( L-\kappa L)}  \int_{\partial B_{L} (\bold{x}_0)}|\widetilde{u}| ds,
$$
and hence, the assertion follows.
\end{proof}

\begin{remark} \label{improradiu}
The refined convergence region for Taylor's series, along with the error estimates within this region, at each point in the open disk in Theorem \ref{holomoresul}, is improved compared to that in \cite{E10} and \cite{HL11}.  Indeed, for each $\bold{x}_0 \in B_R$, the radius of convergence for both the Taylor series and the error estimates at $\bold{x}_0$ in Theorem \ref{holomoresul} is $\frac{1}{3} (R - \|\bold{x}_0\|)$.
On the other hand, the radius for the convergence of the Taylor series and for the error estimates at $\bold{x}_0$ in \cite[Section 2.2, Theorem 10]{E10} and \cite[Theorem 1.14]{HL11} is $\frac{R-\| \bold{x}_0\|}{512e}$ and $\frac{R-\|\bold{x}_0\|}{16 e}$, respectively.
\end{remark}

\centerline{}

\section{Acceleration of uniform convergence}  \label{lastsecacc}
So far, without using multi-variable integral calculus, we show not only the solvability of the main equation \eqref{maineq} with uniform estimates via $L^1$-boundary data but also the uniform convergence by harmonic polynomials for the unique solution to \eqref{maineq} based on the classical results on uniform convergence of Fourier's series. Near the boundary, the convergence rate depending on the H\"{o}lder continuity of $f$ is $O\left( \ln n  \left(\frac{1}{n}\right)^{\alpha} \right)$.  Using the more accelerated uniform convergence of the Fourier series for smooth functions and the arguments we have used in Theorem \ref{limitharmo}, we can achieve better convergence rates than $O\left( \ln n  \left(\frac{1}{n}\right)^{\alpha} \right)$ near the boundary.

\begin{lemma}[{\cite[page 22, Corollary IV]{jack94}}] \label{jackson2}
Let $k \in \mathbb{N}$ and $f \in C^k([-\pi,\pi])$ be periodic with $2\pi$-period. Then, for any $\theta \in [-\pi, \pi]$ and $n \in \mathbb{N}$, it holds that
$$
|f(\theta)-S_n(f)(\theta)| \leq \frac{\gamma_k}{n^k} \cdot \ln n  \cdot  \omega_{f^{(k)}}\left(\frac{2\pi}{n}\right).
$$
where $S_n$ is defined as in \eqref{fouridefn} and $\gamma_k$ is a constant depending on $k$ and
$\omega_{f^{(k)}}$ is defined as
$$
\omega_{f^{(k)}}(\delta) := \sup_{\substack{\theta_1,\theta_2 \in [-\pi,\pi] \\ |\theta_1-\theta_2|\leq \delta}}|f^{(k)}(\theta_1)-f^{(k)}(\theta_2)|, \quad \delta>0.
$$
\end{lemma}

\begin{corollary} \label{higfourcon}
Let $f \in C^{k,\alpha}([-\pi,\pi])$ for some $\alpha \in (0,1]$ and $f$ be periodic with $2\pi$-period. Then, 
$$
|f(\theta)-S_n(f)(\theta)| \leq  (2\pi)^{\alpha} \gamma_k [f^{(k)}]_{C^{\alpha}([-\pi, \pi])}\cdot  \left( \frac{1}{n} \right)^{k+\alpha} \ln n , \quad \text{$\forall \theta \in [-\pi, \pi]$},
$$
where $\gamma_k$ is a constant as in Lemma \ref{jackson2}. 
\end{corollary}
We obtain the following result as a direct consequence of Corollary \ref{higfourcon} and the proof of Theorem \ref{limitharmo}.
\begin{theorem}
Let $g \in C^{\alpha}(\partial B_R)$ for some $\alpha \in (0,1]$ and $f$, $(\widetilde{u})_{n \geq 1}$ and $\widetilde{u}$ be defined in Theorem \ref{limitharmo}. Assume that $f \in C^{k, \alpha}([-\pi, \pi])$ for some $k \in \mathbb{N}$. Then,  for each $n \geq e$
\begin{align*} 
&|\widetilde{u}(x,y)-\widetilde{u}_n(x,y)| \leq  2\gamma_k (2\pi)^{\alpha}[f^{(k)}]_{C^\alpha([-\pi, \pi])} \left(\frac{\sqrt{x^2+y^2}}{R}\right)^{n+1} \left( \frac{1}{n} \right)^{k+\alpha} \ln n, \quad \forall \text{$(x,y) \in \overline{B}_R$},   
\end{align*}
where $\gamma_k$ is a constant as in Lemma \ref{jackson2}. 
\end{theorem}
\noindent
{\bf Acknowledgment.}\;
The author sincerely thanks the reviewers for their valuable comments and suggestions, which improved the quality of this paper.
\centerline{}

\centerline{}
\centerline{}
\noindent
Haesung Lee\\
Department of Mathematics and Big Data Science,  \\
Kumoh National Institute of Technology, \\
Gumi, Gyeongsangbuk-do 39177, Republic of Korea, \\
E-mail: fthslt@kumoh.ac.kr, \; fthslt14@gmail.com \\ \\

\end{document}